\documentclass[12pt]{article}
\usepackage{ct}

\usepackage{enumitem,subcaption}

\specs{14}{4 (1)}{2024}{}

\dateline{Jun 26, 2021}{Feb 20, 2024}{Jun 30, 2024}



\title{A criterion for sharpness\\in tree enumeration and the asymptotic number of triangulations in Kuperberg's $G_2$ spider}

\author[1]{Robert Scherer\thanks{This material is based upon work supported by the National Science Foundation under Grant No. DMS-1800725.}}

\affil[1]{%
Department of Mathematics, University of California, Davis, U.S.A.

\email{rescherer@ucdavis.edu }%
}

\newcommand{\oh}[0]{\mathcal{O}}
\newcommand{\Id}[0]{\operatorname{Id}}
\newcommand{\Arg}[0]{\operatorname{Arg}}

\begin{document}

\maketitle

\begin{abstract}
We prove a conjectured asymptotic formula of Kuperberg from the representation theory of the Lie algebra $G_2$. 
Given two non-negative integer sequences $(a_n)_{n\geq 0}$ and~$(b_n)_{n\geq 0}$, with $a_0=b_0=1$, it is well-known that if the identity $B(x)=A(xB(x))$ holds for the generating functions $A(x)=1+\sum_{n\geq 1} a_n x^n$ and $B(x)=1+\sum_{n\geq 1} b_n x^n$, then $b_n$ is the number of rooted planar trees with $n+1$ vertices such that each vertex having~$i$ children may be colored with any one of $a_i$ distinct colors. Kuperberg proved 
a specific case when this identity holds, namely when $b_n=\dim \operatorname{Inv}_{G_2} (V(\lambda_1)^{\otimes n})$, where $V(\lambda_1)$ is the 7-dimensional fundamental representation of $G_2$, and $a_n$ is the number of triangulations of a regular $n$-gon such that each internal vertex has degree at least $6$. He also observed that $\limsup_{n\to\infty}\sqrt[n]{a_n}\leq 7/B(1/7)$ and conjectured that this estimate is sharp, or, in terms of power series, that the radius of convergence of $A(x)$ is exactly $B(1/7)/7$. We prove this conjecture by introducing a new criterion for sharpness in the analogous estimate for general power series $A(x)$ and $B(x)$ satisfying $B(x)=A(xB(x))$.  Moreover, by way of singularity analysis performed on a recently discovered generating function for $B(x)$, we significantly refine the conjecture by deriving an asymptotic formula for the sequence $(a_n)$.
\end{abstract}


\pagebreak 
\section{Introduction}\label{Introduction} 
\subsection{Motivation from representation theory}
We explore the sharpness of a certain estimate that occurs naturally in the asymptotic enumeration of rooted trees. Our motivation is a particular problem from the literature, a conjecture formulated by Kuperberg in his study of the representation theory of simple rank-$2$ Lie algebras~\cite[Conjecture 8.2]{Kuperberg}. 

Specifically, set $a_0=1$, and for each positive integer $n$, let $a_n$ denote the number of triangulations of a regular $n$-gon, such that the minimum degree of each internal vertex is $6$. The sequence begins
\[
(a_n)_{n=0}^\infty =1, 0, 1, 1, 2, 5, 15, 50, 181, 697,\dots
\]and is indexed in the On-Line Encyclopedia of Integer Sequences (OEIS, \cite{OEIS}) by A059710. 
Next, let $b_0=1$, and for each positive integer $n$, let $b_n$ denote the dimension of the vector subspace of invariant tensors in the $n$-th tensor power of the 7-dimensional fundamental representation of the exceptional Lie algebra $G_2. $ The sequence begins
\[
(b_n)_{n=0}^\infty=1, 0, 1, 1, 4, 10, 35, 120, 455, 1792,\dots
\]
and is indexed in OEIS as A059710.

The sequence $(b_n)$ is also known to have a combinatorial interpretation as the number of lattice walks in the dominant Weyl chamber of the root system for $G_2$ (a $30^\circ$ sector in the triangular lattice in $\mathbb{R}^2$) that start and end at the origin, subject to certain constraints on the steps \cite{Westbury}. This type of model is not unique to $G_2$ or this particular representation. In general, if $V$ is any irreducible representation of any complex semi-simple Lie algebra $L$, there is a similar lattice walk model for the dimension of the space of $L$-invariant $n$-tensors over $V$ \cite[Thm. 5]{Grabiner}. 

Now let $A(x)=1+\sum_{n=1}^\infty a_nx^n$ and $B(x)=1+\sum_{n=1}^\infty b_n x^n$ be the ordinary generating functions for $(a_n)_{n=0}^\infty$ and $(b_n)_{n=0}^\infty$, respectively. In \cite[Section 8]{Kuperberg}, Kuperberg proved the following remarkable identity of formal power series: 
\begin{equation}\label{eqn-functional_equation}
B(x)=A(xB(x)).
\end{equation}
He also observed that $B(x)$ has radius of convergence $1/7$, that $B(1/7)<\infty$, and that, by~\eqref{eqn-functional_equation}, $A(x)$ has radius of convergence at least $(1/7)B(1/7)$,
a constant whose numerical value he estimated to be approximately $6.811.$ He conjectured that this bound is in fact an equality.
\begin{conjecture}[Kuperberg, 1996 \cite{Kuperberg}]\label{conjecture}
\[
\limsup_{n\to\infty} \sqrt[n]{a_n}=7/B(1/7).
\]
\end{conjecture}
It is straightforward to show that the $\limsup$ is in fact a limit (see Section \ref{sec-outline}). We prove here that this conjecture is true and explicitly identify the value of Kuperberg's constant $7/B(1/7)$.   Moreover, we go beyond the exponential growth term to establish a true asymptotic formula for~$(a_n)$, and we derive a full asymptotic expansion for $(b_n)$. The precise result is as follows.
\begin{theorem}\label{thm-conjecture}
Let $A(x)$ and $B(x)$ be as above. Define constants $\rho$, $K$, and $M$ by:
\begin{align}
\rho&=\frac{7}{B(1/7)}\\
K&=\frac{4117715\sqrt{3}}{864\pi}&\approx& \,2627.6\label{eqn-K}\\
M&=\frac{4\sqrt{3}}{421875\pi}\left(\frac{8575\pi-15552\sqrt{3}}{2592\sqrt{3}-1429\pi}\right)^7&\approx&\,1721.0 \label{eqn-M}\\ \nonumber
\end{align}
Then we have the following:
\begin{enumerate}[wide, labelindent=0pt, label=\normalfont]
\item[\normalfont(a)]
\textbf{Kuperberg's conjecture is true.} As $n\to\infty$,
\begin{equation} \label{eqn-first}
a_n=\rho^{n+o(n)}.
\end{equation}
\item[\normalfont(b)]
\textbf{Explicit value of $\rho$.}
The constant $\rho$ has the explicit value
\begin{equation}\label{eqn-rho}
\rho=\frac{5\pi}{8575\pi - 15552\sqrt{3}}\,\approx \,6.8211.
\end{equation}
\item[\normalfont(c)]\textbf{Asymptotic expansion of $b_n$.}
As $n\to \infty$, the sequence $(b_n)$ grows asymptotically as 
\begin{equation}\label{eqn-b_asymptotic}
b_n= K\frac{7^n}{n^7}\left(1+\oh\left(\frac{1}{n}\right)\right).
\end{equation}
Furthermore, there exists a computable sequence of rational numbers $(\kappa_i)_{i=7}^\infty$, with $\kappa_7 \! = \! K\pi/\sqrt{3}$, such that as $n\to\infty$,
\begin{equation}\label{eqn-b_series}
b_n\sim \frac{7^n\sqrt{3}}{\pi} \sum_{i=7}^\infty \frac{\kappa_i}{n^i}.
\end{equation}
\item[\normalfont(d)]\textbf{Asymptotic formula for $a_n$.} 
Conjecture \ref{conjecture} admits the following refinement. As $n\to \infty$,
\begin{equation}\label{eqn-a_asym}
a_n= M\frac{\rho^n}{n^{7}}\left(1+\oh\left(\frac{\log n}{n}\right)\right).
\end{equation}
\end{enumerate}
\end{theorem} 
Furthermore, because it is known (e.g. \cite{Flajolet}, Theorem VII.8) that the dominant asymptotic term of the coefficients of an algebraic power series cannot have a subexponential growth factor of the form $n^{-k}$ for $k$ a positive integer, the estimates \eqref{eqn-b_series} and \eqref{eqn-a_asym} immediately imply the following corollary.

\begin{corollary}
$A$ and $B$ are not algebraic.
\end{corollary}

\subsection{A criterion for sharpness in tree enumeration}

It is interesting to view Conjecture \ref{conjecture} in a more general context, as an asymptotic enumeration problem in the combinatorial theory of rooted trees, where the functional equation \eqref{eqn-functional_equation} has a classical interpretation.
A \textit{planar rooted tree} is an undirected acyclic graph, equipped with a distinguished node and an embedding in the plane (so that distinct subtrees dangling from the same node are ordered amongst themselves). 
Let $A(x)=1+\sum_{n\geq 1} a_n x^n$, where the $a_n$'s are arbitrary non-negative integers, and suppose that $B(x)=1+\sum_{n\geq 1} b_nx^n$ is related to $A(x)$ via~\eqref{eqn-functional_equation}. Then $b_n$ counts the number of planar rooted trees with $n+1$ nodes (including the root), such that for each $i\geq 1$, an internal node having $i$ children can be colored with one of $a_i$ distinct colors, and leaves are not colored. 

If $A(x)$ and $B(x)$ have radii of convergence $R>0$ and $r>0$, respectively, and they satisfy~\eqref{eqn-functional_equation}, then the inequality $rB(r)\leq R$ holds if $a_n\geq 0$ for all $n\geq 1$ (see Lemma \ref{lemma-identity}). It is natural then to ask when equality holds, namely when $rB(r)=R$. We address this question in Section~\ref{section_2} and derive a criterion for equality in the estimate $rB(r)\leq R$. A simplified version of this criterion reads as follows. 
\begin{theorem}[Criterion for sharpness, simplified version]\label{thm-criterion-1}
With $A(x)$, $B(x)$, $R$, and $r$ as in the preceding paragraph, assume that $a_n\geq 0$ for all $n\geq 1$ and that $\operatorname{gcd}(\{n\geq 1: a_n > 0\})=1$. Then
\[
b_nr^n \neq \Theta(n^{-3/2}) \text{ as } n\to\infty \implies R=rB(r). 
\]
\end{theorem}
Conjecture \ref{conjecture} follows from this criterion. Indeed, as we will see in Section \ref{proof}, a formula from the character theory of Lie algebra representations implies the estimate $b_n/7^n =\Theta (n^{-7})$ for the sequence $(b_n)$ in the conjecture.

\subsection{Structure of the paper} 

In {\bf Section \ref{section_2}}, we will derive Theorem \ref{thm-criterion-1} from a more technical version of the sharpness criterion, which is presented as Theorem $\ref{thm-criterion}$. We will also provide two examples, aside from our central example, which is Conjecture \ref{conjecture}. Example \ref{example_1} is classic, and Example \ref{example_2} is new.

In {\bf Section \ref{proof}}, we will prove \eqref{eqn-first}. As mentioned above, in light of Theorem~\ref{thm-criterion-1} the key step is to obtain a growth estimate for the sequence $(b_n)$. We will obtain the estimate \eqref{eqn-b_asymptotic} through saddle-point analysis of a known formula from the character theory of Lie algebra representations (see Proposition~\ref{prop-asymptotics}).

In {\bf Section~\ref{sec-improved}}, we will prove the remaining parts of Theorem \ref{thm-conjecture}, namely the explicit value of $\rho$ in \eqref{eqn-rho} and the asymptotic formulas \eqref{eqn-b_series} and \eqref{eqn-a_asym} for $(b_n)$ and $(a_n)$. The authors of~\cite{Bostan} give a new representation of the generating function $B(x)$ in terms of hypergeometric series. The precise value of $\rho$ will be evaluated from their generating function. Then, we will apply to their generating function an approach called {\it singularity analysis}. In particular, we will appeal to the classical transfer theorems of Flajolet and Odlyzko (see the \nameref{sec-appendix}) to relate the asymptotics of $B(x)$ and $A(x)$ near their respective singularities, $1/7$ and $\rho$, to the asymptotic growth of the coefficient sequences $(b_n)$ and $(a_n)$. A critical step in this process involves showing that~$B(x)$ and $A(x)$ each admit an analytic continuation to a Delta-domain, as defined in the \nameref{sec-appendix}. Another critical step is the derivation of singular expansions for these functions within their Delta-domains. This all will require significant technical work. 
Similar analysis occurred in~\cite{BC} and  \cite{Felsner} in the context of enumerating certain planar graphs. It might be interesting to note that it is precisely the fact that $rB(r)=R$ in our central example that is the source of the subtlety in extending $B(x)$ and $A(x)$ to Delta-domains and establishing their singular expansions, as compared to the situation $rB(r)<R$, which has been more thoroughly studied in the literature and in which case the complex analysis is typically straightforward.
Example~\ref{example_2} and the remark preceeding it may further clarify this perspective.  

In the {\bf \nameref{sec-appendix}}, we review the transfer theorems of Flajolet and Odlyzko and the concept of a Delta-domain, which will be used several times throughout the paper.


\section{Sharpness criterion proof and examples} \label{section_2} 
\subsection{Preliminary  lemmas}
We begin with two lemmas that reframe the combinatorial relationship \eqref{eqn-functional_equation} as an identity of analytic functions. The ideas presented here are well-known in the area of analytic combinatorics and originally arose in the context of tree enumeration. For a full treatment, the interested reader can consult \cite{Flajolet}, particularly sections I.5, VII.4, and VII.5.

Let $A(x)=1+\sum_{n\geq 1} a_nx^n$ and $B(x)=1+\sum_{n\geq 1} b_nx^n$ be power series satisfying \eqref{eqn-functional_equation}. When $B(x)$ has a positive radius of convergence, we would like to know when the identity \eqref{eqn-functional_equation} of formal power series is also an identity of the complex analytic functions defined by these power series in a neighborhood of the origin, since then we may apply analytic methods. A sufficient condition is given by Lemma \ref{lemma-identity} below.

We will adopt a useful convention of setting $y(x):=xB(x)$, whereby the identity \eqref{eqn-functional_equation} can be rewritten as
\begin{equation}
\label{eqn-functional_equation_2}
y(x)=xA(y(x)).
\end{equation}
The coefficient sequence $(y_n)_{n\geq 1}$ of $y(x)=\sum_{n \geq 1}y_nx^n$ is then given simply by $y_n=b_{n-1}$, for~$n\geq 1$. Recall that if $A(x)$ is known, then the coefficients $(y_n)$ are unique and can be determined from the Lagrange Inversion Theorem \cite[Ch. 5.4]{Stanley}: For all $n\geq 1$,
\begin{equation}\label{eqn-Lagrange}
y_n=\frac{1}{n}[x^{n-1}](A(x)^{n}),
\end{equation}
where $[x^n]f(x)$ denotes the $n$'th coefficient of a power series $f(x)$.

When thinking of these power series as analytic functions, we will use the variable $z$ or $w$, e.g. writing $y(z)$; whereas writing $y(x)$ will emphasize thinking of formal power series combinatorially, without considerations of convergence or analytic continuation. This distinction is more psychological than mathematical and could also just be ignored.

\begin{lemma}\label{lemma-identity}
Let $A(x)=1+\sum_{n\geq 1}a_nx^n$ and $y(x)=\sum_{n \geq 1}y_nx^n$ be power series over $\mathbb{R}$ related by \eqref{eqn-functional_equation_2}, with $a_n\geq 0$ for all $n\geq 1$
. Then $y_n\geq 0$ for all $n$. Let $R$ be the radius of convergence of $A$, and let $r$ denote the radius of convergence of $y$. Then $R>0$ iff $r>0$. Moreover,
\begin{equation}\label{eqn-functional_equation_3}
y(z)=zA(y(z))
\end{equation} for $z\in \mathbb{C}$ such that $|z|<r$. Also, $y(r)\leq R$, including when $y(r)=\infty$. \end{lemma} 

We omit a formal proof, but note that the non-negativity of the coefficients of $y(x)$ follows from \eqref{eqn-Lagrange}, and the other facts can be deduced from the implicit function theorem.
Notice that neither $r$ nor $R$ are required to be finite in the lemma.
However, it is easy to see that if $a_n\geq 1$ for all large enough $n$, then $r<\infty$.

As we have already mentioned, the unique solution $y(x)$ to \eqref{eqn-functional_equation_2} is given by the Lagrange Inversion formula \eqref{eqn-Lagrange}. When the coefficients of $A(x)$ are furthermore non-negative, it is common to associate to $A$ another function $\psi$, defined by
\begin{equation}\label{eqn-psi}
\psi(z):=\frac{z}{A(z)},
\end{equation}
on the set $\{z \in \mathbb{C} :|z|<R \text{ and } A(z)\neq 0\}$, where $R$ is the radius of convergence of $A$. 
On the same set, we have the identity
\begin{equation}\label{eqn-derivative}
\psi'(z)=\frac{A(z)-zA'(z)}{A(z)^2}=\frac{1-\sum_{n\geq 1} a_n(n-1)z^n}{A(z)^2}.
\end{equation}
The key point, explained in the Lemma \ref{lemma-psi}, is that $\psi$ is the functional inverse of $y$.

\begin{remark}
While an attempt is made in what follows to state results with a certain amount of generality, a good class of generating functions to have in mind is that for which $(a_n)_{n \geq 1}$ is an eventually positive and eventually increasing sequence of integers, in which case $(y_n)_{n\geq 1}$ has the same property. This includes the functions from Theorem \ref{thm-conjecture}.
\end{remark}

\begin{lemma}\label{lemma-psi}
Let $A(x)=1+\sum_{n\geq 1}a_nx^n$ and $y(x)=\sum_{n \geq 1}y_nx^n$ be power series over $\mathbb{R}$ related by \eqref{eqn-functional_equation_2}, with $a_n\geq 0$ for all $n\geq 1$. Let $R$ be the radius of convergence of $A$, let $r$ denote the radius of convergence of $y$, and assume that $R>0$. Define $\Omega:=\{z\in\mathbb{C}: |z|<r\}$.

Then the function $y:\Omega\to y(\Omega)$ is a biholomorphism, and $\psi(y(z))=z$ for $z\in\Omega$. In particular, $\psi$ is well-defined and analytic on $y(\Omega)$. 
Furthermore, if $A(x)$ is not identically equal to $1$, then $r<\infty$. Finally, if $y(r)<\infty$, then $y$ extends continuously to the boundary $\partial \Omega$, $\psi$ extends continuously to the boundary of $y(\Omega)$, and $\psi(y(z))=z$ holds for $z \in \partial \Omega. $
\end{lemma}
\begin{proof}
Note that $y(0)=0$, $A(0)=1$, and $A$ is analytic on $y(\Omega)$ by Lemma \ref{lemma-identity}. If $y(z)=0$ for~$z\in\Omega$, then by \eqref{eqn-functional_equation_3} we see that 
$
z=z\cdot 1= zA(y(z))=y(z)=0.
$ So $y$ vanishes at the origin and nowhere else in $\Omega$. Therefore, for $z\in\Omega\setminus{\{0\}}$ we have $A(y(z))=y(z)/z\neq 0$, and $A$ does not vanish on $y(\Omega)$. As a consequence, $\psi$ is defined there and analytic, and it follows from~\eqref{eqn-functional_equation_3} that $\psi(y(z))=z$ for $z\in\Omega$. The injectivity of $y$ is now a simple consequence of~\eqref{eqn-functional_equation_3}. Indeed, if $y(z)=y(w)\neq 0$ for $z,w\in\Omega$, then $z\neq 0$, $w\neq 0$, and
\[
zA(y(z))=y(z)=y(w)=wA(y(w))=wA(y(z)),
\]
so $w=z$.

To see that $r<\infty$ in the case that $A(x)\neq 1$, observe that if instead $y$ is entire, i.e.~${\Omega=\mathbb{C}}$, then $R=\infty$ by Lemma $\ref{lemma-identity}$. Moreover, $[0,\infty)=y([0,\infty))\subset y(\mathbb{C})$, since $y$ generally maps~$[0,r)$ onto $[0,y(r))$.
Since $A$ does not vanish on $y(\mathbb{C})$, $\psi'$
is defined on a neighborhood of $[0,\infty)$. Also, $a_n>0$ for some $n>0$, since $A\neq 1$. 
It follows then from the series expansion in \eqref{eqn-derivative} that $\psi'(\tau)=0$ for some $\tau>0$. By the inverse function theorem, we have~$y'(\psi(\tau))=1/\psi'(\tau)=\infty$, which contradicts that $y$ is entire. By this contradiction we conclude that $r$ is finite. 

To prove the last part, $y(r)<\infty$ implies that $\sum_{n=0}^\infty y_nz^n$ converges on $\overline{\Omega}$, and furthermore that this convergence is uniform, by the non-negativity of the $y_n$ and the Weierstrass M-test. This gives the continuous extension of $y$ to $\partial \Omega$. The same argument can be used to show that $A\circ y$ extends continuously to $\partial\Omega$. By continuity, $zA(y(z))=y(z)$ for $z\in\partial\Omega$, and $A$ does not vanish on $\partial (y(\Omega))$ by the same argument as above. So $\psi(y(z))=z$ for $z\in \partial \Omega$. 
\end{proof}

\subsection{Proof of Theorem \ref{thm-criterion-1}}
Recall from Lemma \ref{lemma-identity} that the identity $y(x)=xA(y(x))$ implies that $A$'s radius of convergence~$R$ satisfies the estimate $y(r)\leq R$, where $r$ is the radius of convergence of $y$. The following theorem provides a sufficient condition for $R$ to equal $y(r)$.

\begin{theorem}\label{thm-main}
Suppose that the generating functions $A(x)=1+\sum_{n=1}^\infty a_nx^n$ and $y(x)=\sum_{n=1}^\infty y_n x^n$, with radii of convergence $R$ and $r$ respectively, satisfy the following two conditions: 
\begin{enumerate}
\item[(1)] $R>0$, and $a_n\geq 0$ for $n\geq 1$.

\item[(2)] $A(x)$ and $y(x)$ are related as formal power series by \eqref{eqn-functional_equation_2}.

\end{enumerate}

If $A(z)-zA'(z)$ does not vanish at $z=y(r)$, then $y(r)=R$.
\end{theorem}
\begin{proof} Toward proving the contrapositive, assume that $R>y(r)$. We then need to argue that~$A(z)-zA'(z)$ vanishes at $z \! = \! y(r)$.
We will show that the assumption that ${A(z)-zA'(z) \! \neq \! 0}$ for $z=y(r)$ leads to a contradiction of the fact that $r$ is the radius of convergence of $y$. Note that $r>0$, by Lemma \ref{lemma-identity}, and that $r<\infty$, as otherwise $y(r)=R=\infty$.

Assume that $A(y(r))-y(r)A'(y(r))\neq 0$. Recall the function $\psi$, defined in \eqref{eqn-psi} by $\psi(z)=z/A(z)$. By Lemma \ref{lemma-psi}, $\psi$ is analytic and equal to $y^{-1}$ on $y(\Omega)$, where $\Omega=\{z\in\mathbb{C}:|z|< r\}$. We claim that there exists a small open disk $E$ centered at $y(r)$, such that $\psi$ extends to be analytic on $y(\Omega)\cup E$. Indeed, since $y(r)<R$, it follows that $A$ is analytic on an open disk, which we call~$E$, that is centered at $y(r)$ and contained in the disk $\{z:|z|<R\}$. Also, since $A$ is continuous and $A(y(r))=y(r)/r >0$, we may assume, by possibly replacing $E$ with a smaller open disk, that $A$ does not vanish on $E$. It follows that $\psi$ is analytic on $y(\Omega)\cup E$, as the reciprocal of a non-vanishing analytic function, and furthermore \eqref{eqn-derivative} holds on $y(\Omega)\cup E$. By \eqref{eqn-derivative}, the assumption that $A(y(r))-y(r)A'(y(r))\neq 0$ implies that $\psi'(y(r))\neq 0$. It follows that $\psi$ is locally invertible at $y(r)$. That is, after possibly replacing $E$ with a smaller open disk centered at $y(r)$, we 
see that the map $\psi|_{E}: E\to \psi(E)$ is a homeomorphism, with an analytic inverse map $\psi|_{E}^{-1}$. Since $r$ is a boundary point of $\Omega$, we have by Lemma \ref{lemma-psi} that~${\varnothing\neq \psi(y(\Omega)\cap E)\subset \Omega \cap \psi(E)}$. Since $\psi|_E$ is injective, one sees that $y$ and $\psi|_{E}^{-1}$ agree on the open set $\psi(y(\Omega)\cap E)$. By uniqueness of analytic continuation $\psi|_{E}^{-1}$ also agrees with $y$ on the larger open set $\Omega \cap \psi(E)$, thus acting as an analytic continuation of $y$ to $\psi(E)$. Since $r\in\psi(E)$, this contradicts a fact known as Pringsheim's Theorem \cite[Thm. 5.7.1]{Hille}, which asserts that an analytic function with non-negative real coefficients and a finite radius of convergence necessarily has a singularity at the point where the boundary of its disk of convergence intersects $[0,\infty)$. This is the contradiction we sought, and the proof is complete. 
\end{proof}

If we phrase the theorem in a slightly different form, by replacing the antecedent with the statement that $A(z)-zA'(z)\neq 0$ for \emph{all} $z$ in $(0,R)$,
then the converse  is well-known to be true, and it follows from Theorem \ref{thm-growth}(1) below. Theorem~\ref{thm-growth} contains even deeper asymptotic information than that, however, in particular regarding the \emph{subexponential} (i.e. polynomial) growth rate of $(y_n)$. This, it turns out, will be instrumental in proving Theorem \ref{thm-conjecture}, as it shows how information about the growth of $(y_n)$ can certify that $A(z)-zA'(z)$ does not vanish on~$(0,R)$, and hence, by Theorem \ref{thm-main}, that $y(r)=R$.

\begin{theorem}[Meir, Moon, 1978 \cite{Meir}]\label{thm-growth}
Suppose that $A(x)=1+\sum_{n=1}^\infty a_nx^n$ and $y(x)=\sum_{n=1}^\infty y_n x^n$ (with radii of convergence $R$ and $r$, respectively) satisfy conditions (1) and (2) of Theorem \ref{thm-main}. If there exists $\tau\in(0,R)$, such that $A(\tau)-\tau A'(\tau)=0$, then
\begin{enumerate}
\item[(1)] $y(x)$ has radius of convergence $r=\tau/A(\tau)$, and $y(r)=\tau<R$. 
\end{enumerate}
If, in addition, $\operatorname{gcd}(\{n\geq 1: a_n > 0\})=1$, 
then
\begin{enumerate}
\item[(2)] the coefficient sequence $(y_n)_{n\geq 1}$ satisfies the following asymptotic estimate: As $n\to\infty$,
\begin{align*}
y_n&=\frac{C}{r^nn^{3/2}}(1+\mathcal{O}(n^{-1}))=\frac{C\cdot A'(\tau)^n}{n^{3/2}}(1+\mathcal{O}(n^{-1})),
\end{align*}
where $C =\sqrt{\frac{A(\tau)}{2\pi A''(\tau)}}$. 
\end{enumerate}
\end{theorem}

The theorem's gcd condition is a mild technicality, satisfied for example if $(a_n)$ is eventually increasing. (When the gcd exceeds 1, the asymptotic formula applies, with a modified constant~$C$, to the subsequence of $(y_n)$ consisting of positive values.)
This theorem appears to have been first established essentially by Meir and Moon in \cite[Thm. 3.1]{Meir}, building on techniques of Darboux \cite{Darboux}, P\'{o}lya \cite{Polya}, and others.
In \cite{Meir_2} the same authors generalize their analysis to a much broader class of functional equations, of which \eqref{eqn-functional_equation_2} is an example. One can consult \cite[Thm. 5]{Drmota} and \cite[Thm. VI.6]{Flajolet} for proofs of this more general result, and one can find in \cite[pp.~467-471]{Flajolet} a brief note about the theorem's history. 

Taken together, Theorems \ref{thm-main} and \ref{thm-growth} imply the following dichotomy. 
\begin{theorem}[Criterion for sharpness, full version]\label{thm-criterion}
Suppose that $A(x)$ and $y(x)$ (with radii of convergence $R$ and $r$, respectively) satisfy conditions (1) and (2) of Theorem \ref{thm-main} and the gcd condition of Theorem \ref{thm-growth}. Then exactly one of the following is true: 
\begin{enumerate}
\item[(1)] $A(z)-zA'(z)$ is non-vanishing for $z\in(0,R)$, and $y(r)=R$.
\item[(2)] $A(z)-zA'(z)$ vanishes at $z=y(r)$, and $y(r)<R$. Moreover, $z=y(r)$ is the unique solution to $A(z)-z A'(z)=0$ on $(0,R)$, and $\displaystyle{y_n=Cr^{-n}n^{-3/2}(1+o(1))}$ as $n\to\infty$, for some constant  $C>0$. 

\end{enumerate}
\end{theorem}
\begin{proof}
$A(z)-zA'(z)$ either vanishes on $(0,R)$ or it does not. If it does not, then Theorem \ref{thm-main} implies that $y(r)=R$, which is case (1) of the dichotomy. 
If it does vanish, then Theorem \ref{thm-growth} applies, and case (2) is implied.
\end{proof}

In particular, the existence of the subexponential factor $n^{-3/2}$ in the asymptotic expansion of~$y_n$ is implied  when $y(r) \! < \! R$. The absence of the $n^{-3/2}$ factor therefore certifies that~${y(r)=R}$, and Theorem~\ref{thm-criterion-1} immediately follows.

\subsection{Examples}\label{examples}
The generating functions $A(x)$ and $B(x)$ of Theorem \ref{thm-conjecture} illustrate case (1) of Theorem \ref{thm-criterion} and comprise the main example of this paper. Here we give two further examples related to Theorem~\ref{thm-criterion}. Example \ref{example_1} is a classical example of case (2), while Example \ref{example_2} considers the boundary case in the theorem, namely when $A(z)-zA'(z)$ vanishes at $z=R$. This example falls under case (1) and was manufactured specifically to explore the question of whether the statement in case (2) admits a converse. In particular, assuming conditions (1) and (2) of Theorem \ref{thm-main}, if $y_n\sim Cr^{-n}n^{-3/2}$ for a constant $C$, is the inequality $y(r)\leq R$ necessarily strict? In other words, while case (2) of Theorem \ref{thm-criterion} includes the factor $n^{-3/2}$, the absence of which certifies that $y(r)=R$, does the presence of such a factor imply case (2)? The generating function $A(x)$ that we present in Example \ref{example_2} serves as a counterexample to show that the answer is no. We note that aside from its general interpretation in terms of trees and vertex colorings that we have already discussed, Example \ref{example_2}  does not possess an additional combinatorial interpretation that we are aware of, and it is not presently clear how the example should generalize to a broader class of boundary case examples with natural combinatorial significance. As discussed further in the remark preceding it, Example \ref{example_2} will require significantly more technical analysis than Example \ref{example_1} because of the boundary situation $y(r)=R$.
 
\begin{example}[Catalan numbers]\label{example_1}
Let $A(x)=1/(1-x)$, with radius of convergence $R=1$, and let $y(x)=\sum_{n\geq1} y_nx^n$ satisfy the equation $y(x)=xA(y(x))$.
Then for $n\geq 1$,  $y_{n+1}=\genfrac(){0pt}{2}{2n}{n}/(n+1)$ is the $n$'th Catalan number. One way to show this (e.g. \cite[pp. 1-4]{Drmota}) is to solve $y(x)^2-y(x)+x=0$ for $y(x)$ and then develop a Taylor expansion from the solution
\[
y(x)=\frac{1-\sqrt{1-4x}}{2}.
\]
We see that the radius of convergence of $y(x)$ is $r=1/4$, that $y(r)=1/2<R$, and that
\[A(x)-xA'(x)=\frac{1-2x}{(1-x)^2}\]
 vanishes at $y(r)$. 
Thus, we can conclude that $y_n\sim C\cdot4^{n}n^{-3/2}(1+o(1))$ for some constant C. From Theorem \ref{thm-growth}, the value of $C$ is found to be $4^{-1}\pi^{-1/2}$.
\end{example}

\begin{remark}
Although the Lagrange Inversion Formula \eqref{eqn-Lagrange} provides a method to compute~$(y_n)$ exactly from $A(x)$, it does not make it easy to determine the asymptotic growth of $(y_n)$ in the following example. Instead, the transfer method from the \nameref{sec-appendix} will be necessary. An interesting feature of Example \ref{example_2} will be the extra labor we must go through to justify that $y$ extends to a Delta-domain. This is due to having equality in the universal estimate $y(r)\leq R$. Indeed, when $y(r)<R$, as in Example \ref{example_1}, then $A$ and $\psi$ are a priori analytic in a neighborhood of $y(r)$, and this makes the analytic extension of $y$ to a Delta-domain easier to argue in the majority of applications of the transfer theorems in the literature. 
As with Example \ref{example_2}, our central example (Kuperberg's generating function for triangulation counts) contributes a new instance of the atypical setup $y(r)=R$ and more ad hoc arguments to justify Delta-domain continuation based on idiosyncrasies of the functions at hand (see Section \ref{sec-improved}).
\end{remark}

\begin{example}[Boundary case]\label{example_2} 
 Let 
\begin{align*}
A(x)&=6x+2(1-4x)^2-(1-4x)^{5/2}\\
&=1+2x^2+20x^3+10x^4+12x^5+20x^6+\cdots,
\end{align*}
where the radical is in terms of the principal logarithm. Then $A(x)$ has radius of convergence~$R=1/4$, and the Taylor coefficients of $(1-4x)^{5/2}=\sum_{n=0}^\infty x^n(-4)^n \genfrac(){0pt}{2}{5/2}{n}$
are easily seen to be integers (the coefficients of $(1-4x)^{1/2}$ are integers in the previous example), and are negative except for the constant term. Let $y(x)$ be the unique solution to \eqref{eqn-functional_equation_2}, e.g. as determined by Lagrange Inversion \eqref{eqn-Lagrange}, with radius of convergence $r$, to be determined. Then one sees that conditions (1) and (2) of Theorem \ref{thm-main} are met.

Observe that $A(z)-zA'(z) \! \neq \! 0$ for $z \! \in \! (0,R)$, so Theorem \ref{thm-criterion}(1) implies that ${y(r) \! = \! R \! = \! 1/4}$, and also note that $\lim_{z\to R}[A(z)-zA'(z)]=0$, so this is a boundary case. Define $\psi(z)=z/A(z)$ for $|z|<R$ where $A(z)\neq 0$. By Lemma \ref{lemma-psi}, the identity $z=\psi(y(z))$ remains valid at $z=r$, so \[r=\psi(R)=\psi(1/4)=1/6.\]
We observe that
\begin{equation}\label{eqn-example_a}
\frac{1}{6}-\psi(z)=\frac{2(1-4z)^2-(1-4z)^{5/2}}{6A(z)}=(1-4z)^2H(z),
\end{equation}
where $H(z)$ is analytic and non-vanishing on a neighborhood of $R=1/4$ in the slit plane $\displaystyle{\mathbb{C}\setminus{[1/4,\infty)}}$, and $\lim_{z\to 1/4} H(z)=2/9$.

Now we derive asymptotics for the coefficient sequence $(y_n)$ by using Theorem \ref{thm-transfer} in the \nameref{sec-appendix}. To initiate this process we must justify that the function $y$ admits an analytic continuation to a Delta-domain. We will use the fact that $A$ and $y$ are algebraic.
Let
\[
f(y,z)=1024z^{2} y^{5} - 256z^{2} y^{4} + 68z^{2} y^{2} - 64zy^{3} - 20z^{2}y + 20zy^{2} + 3z^{2} - 4zy + y^{2}.
\]
Then one may check that $f(y,\psi(y))=0$ for all $y$ where $\psi(y)$ is defined and non-zero, so in particular on the open set $\{y:|y|<R, y\neq 0\}$.   It follows that the function $y$, defined initially on~$\Omega=\{z:|z|<1/6\}$ by the power series above and being inverse to $\psi$ near 0, is an analytic solution to $f(y(z),z)=0$ on $\Omega$. Furthermore, $f$ is analytic on $\mathbb{C}\times(\mathbb{C}\setminus{\{0\}})$, and the critical points of $f$ (points where both $f$ and $\frac{\partial f}{\partial{y}}$ vanish) can be computed as the roots of the discriminant of $f$ with respect to $y$ \cite[Ch. IV, Sec 8]{Lang}. Using SAGE, we find that the discriminant has roots at~$z=0$, $z=1/6$, $z\approx -0.02$, as well as two complex conjugate roots~$z\approx -0.12 \pm 0.18i$ of modulus larger than $1/6$. As an algebraic function, $y$ extends analytically to any simply-connected domain $\tilde{\Omega}\supset\Omega$ with the following properties: the coefficients of the minimal polynomial of $y(z)$ over $\mathbb{C}(z)$ are analytic (in this case the minimal polynomial is~$f(y,z)/1024z^2$), and $\tilde{\Omega}$ contains no roots of the discriminant of $f$  \cite[p.119]{Smith}. Furthermore, $f(y(z),z)=0$ for all $\tilde{\Omega}$. An example of such a domain is depicted in Figure \ref{fig1}. In particular, $y$ extends to a Delta-domain. 

\begin{figure}[!ht]
\centering
\includegraphics[scale=0.5]{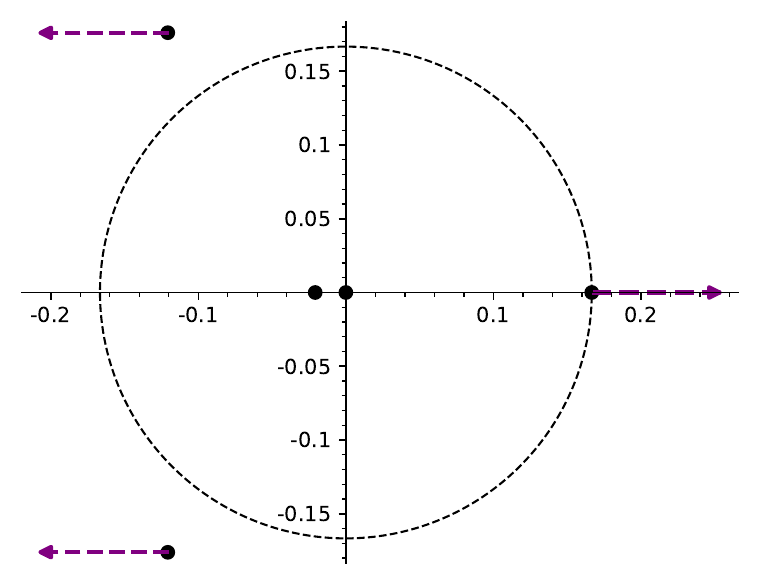}
\caption{The boundary of the disk $\Omega$ and branch cuts whose complement $\tilde{\Omega}$ is a domain of analytic continuation for $y$ and contains a Delta-domain around $\Omega$. Roots of the discriminant of~$f$ are also shown.}
\label{fig1}
\end{figure}

We claim moreover that $y(z)\not\in\mathbb{C}\setminus{[1/4,\infty)}$ for $z\in\tilde{\Omega}$. Indeed, if we view $f(y,z)$ as quadratic over $z$, and collect the coefficients in $\mathbb{C}[y]$ to write $f(y,z)=a(y)z^2+b(y)z+c(y)$, then $f(y,z)$ has discriminant
\[\operatorname{disc}(y)=-\frac{1}{4}(4y - 1)^7 -\frac{1}{2}(4y - 1)^6 - \frac{1}{4}(4y - 1)^5,\]
making it clear that $f(1/4,z)$ has a unique root of 1/6 and that for $y>1/4$, $f(y,z)$ has two complex conjugate roots $z$ and $\bar{z}$. 
Moreover,
one can compute by the quadratic formula that  {\small
\begin{align*}
|z|^2-\frac{1}{36}&=\frac{b(y)^2+\operatorname{disc}(y)}{4a(y)^2}-\frac{1}{36}=\frac{9b(y)^2+9\operatorname{disc}(y)-a(y)^2}{36a(y)^2}\\
&=\frac{-(4y - 1)^{10} - 8(4y - 1)^9 -\cdots - \frac{81}{2}(4y - 1)^3 - \frac{27}{2}(4y - 1)^2}{36(4y - 1)^{10} + 288(4y - 1)^9 +\cdots + \frac{4131}{2}(4y-1)^2+\frac{2916}{4}(4y-1)+\frac{729}{4}},
\end{align*}
}\\
i.e. the non-zero Taylor coefficients of the numerator and denominator are all negative and all positive, respectively. This implies that for $y>1/4$, one has $|z|<1/6$. But we know on the other hand that $|y(z)|<1/4$ for $|z|<1/6$ -- by the non-negativity of $(y_n)$, the fact that~$y(1/6)=1/4$, and the triangle inequality. These two facts imply, since $f(y(z),z)=0$ for all $z\in\tilde{\Omega}$, that $y(z)\not\in [1/4,\infty)$ for $z\in\tilde{\Omega}$. 
In other words, the values of $y$ on $\tilde{\Omega}$ avoid the branch cut $[1/4,\infty)$.

\enlargethispage{.4cm} 
As a result, we may substitute $y(z)$ for $z$ in \eqref{eqn-example_a}, apply the identity $\psi(y(z))=z$, and solve for $y(z)$. We find that
\begin{equation*}\label{eqn-example_b}
y(z)-\frac{1}{4}= -\frac{1}{4(H(y(z)))^{1/2}} \left(\frac{1}{6}-z\right)^{1/2},
\end{equation*}
for $z\in\tilde{\Omega}$.
It follows that  \[y(z)-\frac{1}{4}\sim \frac{-3}{4\sqrt{2}} \left(\frac{1}{6}-z\right)^{1/2},\]
as $z\to1/6$ in $\tilde{\Omega}$. 
We conclude, by Theorem \ref{thm-transfer} in the  \nameref{sec-appendix}, that
\[
y_n\sim\frac{K\cdot 6^n}{n^{3/2}},
\]
as $n\to\infty$, where $K=\frac{\sqrt{3}}{16\sqrt{\pi}}$.
\end{example}


\section{Proof of Conjecture \ref{conjecture}} 
\label{proof} 

For the remainder of the paper, the generating functions $A(x)$ and $B(x)$ are those from Theorem~\ref{thm-conjecture}. Recall that $1/7$ is the radius of convergence of $B(x)$ and $y(x)$, and that $\rho=7/B(1/7)=1/y(1/7)$, and let $R$ denotes the radius of convergence of $A(x)$. Note that Theorem \ref{thm-criterion-1} applies to $A(x)$ and $B(x)$, since it is clear that $(a_n)_{n\geq 1}$ is increasing. Indeed, given any triangulated $n$-gon ($n\geq 3$), one can insert a new external vertex to obtain a triangulated $(n+1)$-gon, without introducing any new internal vertices.

\subsection{Outline of the proof}\label{sec-outline}
Our proof of \eqref{eqn-first} proceeds by the following steps.
\begin{enumerate}
\item[(1)] With $K$ as in Theorem $\ref{thm-conjecture}$, we show in Proposition \ref{prop-asymptotics} that as $n\to \infty$, 
\[
b_n\sim K\frac{7^n}{n^7}.
\]
\item[(2)] Since the sub-exponential growth factor is $n^{7}$, and not $n^{3/2}$, we conclude by Theorem \ref{thm-criterion-1} that $R=1/\rho$.

\item[(3)] It follows that
\begin{equation*}\label{eqn-limsup}
\limsup_{n\to \infty} \sqrt[n]{a_n}=\rho.
\end{equation*}
Observe that $a_{n+m-2}\geq a_n a_m$, for $n,m\geq 2$, since one can always obtain a triangulated $(n+m-2)$-gon by gluing together a triangulated $n$-gon and a triangulated $m$-gon along a common edge, and this process does not introduce any new internal vertices. Thus we obtain, for $n,m\geq 2$,
\[
\log a_n +\log a_m \leq \log a_{n+m-2} \leq \log a_{n+m}.
\]
So $(\log a_n)_{n\geq 2}$ is a superadditive sequence, which implies by a lemma generally attributed to Fekete \cite{Fekete} that
\[
\lim_{n\to\infty}\sqrt[n]{a_n}=\sup_{n\in\mathbb{N}} \sqrt[n]{a_n}.
\]
This limit must be $\rho$, by comparing with the $\limsup$ above, and \eqref{eqn-first} follows directly. 
\end{enumerate}
\subsection[]{Asymptotics of $(b_n)$}
In view of the proof outline from above, the following proposition completes the proof of \eqref{eqn-first}. Let $K$ be as defined in Theorem \ref{thm-conjecture}.
\begin{proposition}\label{prop-asymptotics}
As $n\to\infty$, the sequence $(b_n)_{n=0}^\infty$ satisfies
$
b_n\sim K(7^n/n^7).
$
\end{proposition}
In the introduction we described a lattice walk model in which the $b_n$'s denote the number of $n$-step excursions that start and end at the origin in the Weyl chamber for $G_2$. This interpretation suggests that a local central limit theorem will apply for the return probabilities of random paths of length $n$. A difficulty is that the constraints on allowable steps imply that they are not i.i.d. Luckily, there is a reflection trick due essentially to \cite{Gessel} which allows one to express the return probabilities as local linear combinations of the probabilities that \emph{unconstrained} walks in the weight lattice ($\cong \mathbb{Z}^2$) will end at various nearby points to the origin. Indeed, this method has been applied to derive asymptotic formulas \cite[Thm. 8]{Tate} which imply that the sequence ($\text {dim Inv}_L(V^{\otimes n}))_{n=1}^\infty$, where $V$ is any representation of a complex semi-simple Lie algebra $L$, is asymptotically equivalent to $C\dim(V)^n n^{-\alpha}$, where $\alpha$ is half the dimension of $L$ (in our case, $\alpha=14/2=7$), and the constant term $C$ can also be computed but depends on the specific representation. We nonetheless supply a direct proof of Proposition \ref{prop-asymptotics} based on a saddle-point analysis. Similar analysis for other Weyl chambers was conducted in \cite{Feierl}.

\begin{proof}
The random walk model and the reflection trick mentioned above are encoded by the following formula from character theory \cite[p.15]{Greg}: $b_n$ is the coefficient of $x^ny^n$ in the Laurent polynomial $WM^n$, where
\[
M(x,y)=1+x+y+xy+x^{2}y+xy^{2}+(xy)^{2},
\]
and
\begin{align*}
W(x,y)=x^{-2}y^{-3}(x^2y^3&-xy^3+x^{-1}y^2-x^{-2}y+x^{-3}y^{-1}-x^{-3}y^{-2}\\
&+x^{-2}y^{-3}-x^{-1}y^{-3}+xy^{-2}-x^2y^{-1}+x^3y-x^3y^2).
\end{align*}
We define 
\[f(x,y)=\log\left(\frac{1}{7}M(x,y)\right)-\log(x)-\log(y).
\]

By Cauchy's residue formula for Taylor coefficients, we have
\[
\frac{b_n}{7^n}=\frac{1}{(2\pi i)^2}\oint\oint \left[W(z_1,z_2)\cdot \frac{M(z_1,z_2)^n}{7^n}\cdot \frac{1}{(z_1z_2)^{n+1}}\right] \,dz_1\,dz_2.
\]
We use as contours for both integrals the unit circle about the origin, i.e. $z_1=e^{iu}$ and $z_2=e^{iv}$, for $-\pi \leq u,v\leq \pi$. Thus,
\begin{equation}\label{eqn-countour}
\frac{b_n}{7^n}=\frac{1}{4\pi^2}\int_{-\pi}^\pi\int_{-\pi}^\pi \left[W(e^{iu},e^{iv})\cdot\exp\left(n f(e^{iu},e^{iv})\right)\right]\, du\,dv.
\end{equation}
The function $f$ satisfies $f(1,1)=f_x(1,1)=f_y(1,1)=0$, and thus we have 
\[
f(x,y)=\frac{2}{7}(x-1)^2+\frac{2}{7}(y-1)^2+\frac{2}{7}(x-1)(y-1)+\mathcal{O}((x-1,y-1)^3),
\]
as $(x,y)\to (1,1)$, where the exponent $3$ in the last term signifies a multi-index that ranges over all pairs of non-negative integers whose sum is $3$. It follows that
\[
f(e^{iu},e^{iv})=-\frac{2}{7}u^2-\frac{2}{7}v^2-\frac{2}{7}uv+\mathcal{O}((u,v)^3),
\]
as $(u,v)\to(0,0)$. Now we make an $n$-dependent change of variables in \eqref{eqn-countour}, namely $p=\sqrt{n}u$, and $q=\sqrt{n}v$. 
The integral becomes
{\small{
\begin{equation*}
\frac{4\pi^2b_nn}{7^n}=\int_{-\sqrt{n}\pi}^{\sqrt{n}\pi}\int_{-\sqrt{n}\pi}^{\sqrt{n}\pi} 
\left[W\left(e^{ip/\sqrt{n}},e^{iq/\sqrt{n}}\right)\exp\left(-\frac{2}{7}p^2-\frac{2}{7}q^2-\frac{2}{7}pq+\mathcal{O}\left(\frac{(p,q)^3}{\sqrt{n}}\right)\right)\right]dpdq.
\end{equation*}
}}

By standard estimates, which we omit (and which morally relate to the fact that there is a local central limit theorem for a lattice random walk with i.i.d. steps at work), the above integral can be approximated asymptotically by the completed integral over all of $\mathbb{R}^2$, and the contribution of the $\mathcal{O}((p,q)^3/\sqrt{n})$ term is negligible. Moreover, if for each $k\in\mathbb{N}$ we  let $T_k(p,q)$ denote the order $k$ Taylor approximation for $(p,q)\mapsto W(e^{ip},e^{iq})$, then it suffices to consider integrals of the form 
\[
\int_{-\infty}^\infty \int_{-\infty}^\infty \left[T_{k}(pn^{-1/2},qn^{-1/2})\cdot\exp\left(-\frac{2}{7}p^2-\frac{2}{7}q^2-\frac{2}{7}pq\right)\right].
\]

Computing with SAGE, we find that this integral vanishes for $k<12$, while
\begin{align*}
&\quad\int_{-\infty}^\infty \int_{-\infty}^\infty \left[T_{12}(pn^{-1/2},qn^{-1/2})\cdot\exp\left(-\frac{2}{7}p^2-\frac{2}{7}q^2-\frac{2}{7}pq\right)\right]dpdq\\
&=\frac{1}{n^6}\int_{-\infty}^\infty \int_{-\infty}^\infty \left[T_{12}(p,q)\cdot\exp\left(-\frac{2}{7}p^2-\frac{2}{7}q^2-\frac{2}{7}pq\right)\right]dpdq\\
&=\frac{1}{n^6}\cdot
\frac{4117715\sqrt{3}}{216}\pi.\\
\end{align*}

In total, we have shown that
$\displaystyle{
\frac{4\pi^2b_nn}{7^n}\sim\frac{4117715\sqrt{3}}{216n^6}\pi,}
$
as $n\to\infty$,
which verifies the asserted value of $K$. 
\end{proof}


\section{Proof of Theorem \ref{thm-conjecture}} 
\label{sec-improved}
Having already proven Conjecture \ref{conjecture}, which is part (a) of Theorem \ref{thm-conjecture}, in this section we complete the proof of Theorem \ref{thm-conjecture}, parts (b)--(d).

Recall the functions $y(z)=zB(z)$ and $\psi(z)=z/A(z)$, where $A$ and $B$ be are still the functions from Theorem \ref{thm-conjecture},  and recall that $1/\rho = y(1/7)$. We will use these function names to also indicate analytic continuations to larger domains than their original disks of convergence. Here is an outline of the main steps of the proof, with brief descriptions. Each step has its own dedicated subsection below. 

\begin{enumerate}
    \item[(1)] {\bf Evaluation of $\rho$ and analytic continuation of $B$.} We will verify \eqref{eqn-rho} in Proposition~\ref{prop-evaluation}. In particular, we will use the generating function for $B$ that was discovered in \cite{Bostan} and is copied here in Theorem \ref{thm-Bostan} in order to evaluate $\rho$. Moreover, we will establish an analytic continuation of $B$ to the domain
    \[\tilde{\Omega}:=\mathbb{C}\setminus{\big((\infty, -1/2]\cup [1/7, \infty)\big)}.\]
     Note that $\tilde{\Omega}$ includes a Delta-domain, as defined in the \nameref{sec-appendix}, which is a prerequisite for applying the method of asymptotic transfer indicated in Theorem \ref{thm-transfer}.
    \item[(2)] {\bf Singular expansion of $B$ and asymptotic formula for $(b_n)$.} In Proposition \ref{prop-B_sing}, we will expand $B$ in terms of a logarithm near the singularity $1/7$. Having already extended $B$ to a Delta-domain in the previous step, this will allow us to prove \eqref{eqn-b_asymptotic} and \eqref{eqn-b_series} by the method of asymptotic transfer.
    \item[(3)] {\bf Analytic continuation of $\psi$.} To complete the proof of Theorem \ref{thm-conjecture}, what will remain at this point is to verify \eqref{eqn-a_asym}. The first step is to show that $\psi$ extends analytically to a Delta-domain $\Delta_{1/\rho}$ (using notation from the \nameref{sec-appendix}), and that $(y\circ\psi)|_{\Delta_{1/\rho}}=\Id|_{\Delta_{1/\rho}}$. This is done in Proposition \ref{prop-psi} and serves as the analytic precondition for the next step. 
    \item[(4)] {\bf Singular expansion of $\psi$ near the singularity $1/\rho=y(1/7)$.} In Proposition \ref{prop-psi_sing}, we will apply a bootstrapping procedure to locally invert $y$ near the singularity $1/7$, up to an asymptotically negligible error term. In doing so, we will derive a singular expansion of~$\psi$ near the singularity $1/\rho$. 
    \item[(5)] {\bf Singular expansion of $A$ near the singularity $1/\rho=y(1/7)$.} In Proposition~\ref{prop-A_sing}, we will use the singular expansion of $\psi$ derived in the previous step to obtain a singular expansion for $A$. 
    \item[(6)] {\bf Asymptotic formula for $(a_n)$.} We will verify \eqref{eqn-a_asym} by applying the method of asymptotic transfer from Theorem~\ref{thm-transfer} to the singular expansion of $A$ derived in the previous step. 
    \end{enumerate}
\subsection[]{Evaluation of $\rho$ and analytic continuation of $B$}
\label{sub-ext}
In the recent paper \cite[p. 8]{Bostan} is given the following remarkable closed formula for the generating function $B$ in terms of hypergeometric series. 
\begin{theorem}[Bostan, Tirrell, Westbury, Zhang, 2019]\label{thm-Bostan}
\begin{equation}\label{eqn-formula}
B(z)=\frac{1}{30z^5}\left[ R_1(z)\cdot {_2F_1}\left(\frac{1}{3},\frac{2}{3};2;\phi(z)\right)+R_2(z)\cdot {_2F_1}\left(\frac{2}{3},\frac{4}{3};3;\phi(z) \right)+5P(z)\right],
\end{equation}
where 
\begin{align*}
R_1(z)&=(z+1)^2(214z^3+45z^2+60z+5)(z-1)^{-1},\\
R_2(z)&=6z^2(z+1)^2(101z^2+74z+5)(z-1)^{-2},\\
\phi(z)&=27(z+1)z^2(1-z)^{-3},\\
P(z)&=28z^4+66z^3+46z^2+15z+1.
\end{align*} 
\end{theorem}
We will use this formula as the starting point for our analysis in the next section. As a quick application, we obtain the following result, which verifies \eqref{eqn-rho}.
\begin{proposition}\label{prop-evaluation}
\[
\rho=\frac{5\pi}{8575\pi - 15552\sqrt{3}}.
\]
\end{proposition}
\begin{proof}
Evaluating \eqref{eqn-formula} at $x=1/7$, the formula simplifies to
\[
B\left(\frac{1}{7}\right)=\frac{7^5}{30}\left[\frac{-55296}{2401}\cdot {_2F_1}\left(\frac{1}{3},\frac{2}{3};2;1\right)+\frac{9216}{2401}\cdot {_2F_1}\left(\frac{2}{3},\frac{4}{3};3;1\right)+\frac{150}{7}\right].
\]
It is now a matter of routine calculation to deduce the value in the proposition. One only needs standard facts about the gamma function, namely that $\Gamma(z+1)=z\Gamma(z)$ for $z\not\in\mathbb{Z}_{\leq 0}$ and the following \cite{Bailey}:
\begin{itemize}
\item[(1)] \hfill 
\(\displaystyle
{_2F_1}\left(a,b;c;1\right)=\frac{\Gamma(c)\Gamma(c-a-b)}{\Gamma(c-a)\Gamma(c-b)}  \quad(\operatorname{Re}(c)>\operatorname{Re}(a+b)),
\) \hfill \phantom{.}
\item [(2)] \hfill
\(\displaystyle
\Gamma(z)\Gamma(1-z)=\frac{\pi}{\sin(\pi z)} \quad(z\in\mathbb{C}).
\) \hfill \phantom{.}

\end{itemize}
Using (1) and (2) we simplify the above expression for $B(1/7)$ and recall that $\rho=7/B(1/7)$. 
\end{proof}

In deriving formula \eqref{eqn-formula}, the authors of \cite{Bostan} demonstrate that $B$ is the solution of a linear differential equation of the form
\[
B'''+a_2B''+a_1B'+a_0B=0,
\] where the coefficients $a_i$, $i=0,1,2$, are rational functions with poles at $0, -1/2, -1,$ and $1/7$. From this fact, as well as the fact that $B$ is expressed on $\Omega=\{z:|z|<1/7\}$ by a convergent power series (the generating function from Theorem \ref{thm-conjecture}) and is visibly analytic near $z=-1$ (by \eqref{eqn-formula} and $\phi(-1))=0$), the theory of differential equations implies that $B$ can be continued analytically and uniquely along any path avoiding the set $\{-1/2,1/7\}$ (e.g. \cite[p.119]{Smith}). Therefore, $B$ has a unique analytic continuation to the simply-connected doubly slit plane 
\[
\tilde{\Omega}=\mathbb{C}\setminus{\big((\infty, -1/2]\cup [1/7, \infty)\big)}.
\]

In particular, $B$ is continuable to a Delta-domain around $\Omega$.

\subsection[]{Singular expansion of $B$ and asymptotic formula for $(b_n)$}

In this section we expand $B$ in terms of a logarithm near the singularity $1/\rho$ and then derive the asymptotic formulas \eqref{eqn-b_asymptotic} and \eqref{eqn-b_series} by transfer methods. 
For the remainder of Section \ref{sec-improved}, the principal branch of the logarithm is denoted by $\log$.
\subsubsection[]{Singular expansion of $B$}
\begin{proposition}
With $Z=1-7z$ and $K$ as in \eqref{eqn-K}, \label{prop-B_sing}
\begin{align}\label{eqn-B_sing}
B(z)&=p(Z)-\frac{K}{6!}Z^6\log Z+Z^7H_2(Z)+Z^7H_1(Z)\log Z \\
&=p(Z)-\frac{K}{6!}Z^6\log Z+\mathcal{O}(Z^7\log Z), \nonumber 
\end{align}
as $z\to 1/7$ in $\tilde{\Omega}$, 
where $H_1(Z)$ and $H_2(Z)$ are power series with positive radii of convergence and non-zero constant terms, and $p(Z)$ is a degree-six polynomial with $p(0)=B(1/7)=7/\rho$. The error term $Z^7H_2(Z)+Z^7H_1(Z)\log Z$ in \eqref{eqn-B_sing} extends analytically to $\tilde{\Omega}$.
\end{proposition}

To prove the proposition we will appeal to the following analytic continuation, which is distilled from \cite{Wolfram}, for the $_2F_1$ functions appearing in \eqref{eqn-formula}.
\begin{lemma}\label{lemma-expansion}
For constants $a,b\in\mathbb{R}\setminus{\mathbb{Z}_{\leq 0}}$ and a variable $z$ satisfying $|1-z|<1$, we have
\begin{equation}\label{eqn-1}
_2F_1(a,b;a+b+1,z)=C_{a,b}+S_{a,b}(z)+\log(1-z)\cdot T_{a,b}(z),
\end{equation}
with the following definitions:
\[
C_{a,b}=\frac{\Gamma(a+b+1)}{\Gamma(a+1)\Gamma(b+1)},
\]
\[
T_{a,b}(z)=\frac{\Gamma(a+b+1)}{\Gamma(a)\Gamma(b)}\cdot\left(\sum_{k=0}^\infty \left[\frac{(a+1)_k(b+1)_k}{k!(k+1)!}\cdot(1-z)^{k+1}\right]\right),
\]
and
\[
S_{a,b}(z)=\frac{\Gamma(a+b+1)}{\Gamma(a)\Gamma(b)}\cdot
\left(\sum_{k=0}^\infty \left[ \frac{(a+1)_k(b+1)_k}{k!(k+1)!}\cdot c_k\cdot (1-z)^{k+1}\right]\right),
\]
where
\[c_k=\psi_0(a+k+1)+\psi_0(b+k+1)-\psi_0(k+1)-\psi_0(k+2),\] for the digamma function $\psi_0=\Gamma'/\Gamma,$  $(q)_k=q(q+1)\cdots (q+k-1)$ for $k>0$, and $(q)_0=1$.
\end{lemma}
\begin{proof}[Proof of Proposition \ref{prop-B_sing}]
The second statement regarding analyticity of the error term is immediate from the first statement, since the other summands in \eqref{eqn-B_sing} are analytic on $\tilde{\Omega}$. To prove the first statement, we use Lemma \ref{lemma-expansion} to expand the hypergeometric functions in \eqref{eqn-formula}, obtaining

\begin{equation}\label{B-expansion}
B(z)=f(z)+\log(1-\phi(z))g(z),
\end{equation}
for $|1-\phi(z)|<1$, 
where
\begin{align*}
f(z)=\frac{1}{30z^5}\left[R_1(z)\left(C_{\frac{1}{3},\frac{2}{3}}+S_{\frac{1}{3},\frac{2}{3}}(\phi(z))\right)+R_2(z)\left(C_{\frac{2}{3},\frac{4}{3}}+S_{\frac{2}{3},\frac{4}{3}}(\phi(z))\right)\right]+P(z),
\end{align*}
and
\[
g(z)=\frac{1}{30z^5}\left[R_1(z)\left(T_{\frac{1}{3},\frac{2}{3}}(\phi(z))\right)+R_2(z)\left(T_{\frac{2}{3},\frac{4}{3}}(\phi(z))\right)\right].
\]

The next step in finding a singular expansion for $B$ is to expand $f$ and $g$ in powers of $(1-7z)$. This simply amounts to a Taylor expansion, but for convenience we use SAGE. 
We find, as in our saddle-point analysis from Proposition \ref{prop-asymptotics}, significant cancellation of lower-order terms. Specifically, we have the following Taylor expansions, convergent in a neighborhood of $1/7$:
\[
f(z)=\sum_{n\geq 0}^\infty f_n(1-7z)^n,
\]
where
\[f_0=
7\cdot\frac{85575-15552\sqrt{3}}{5\pi};\]
and 
\[
g(z)=\sum_{n\geq 6}^\infty g_n(1-7z)^n,
\]
where
\begin{equation}\label{eqn-g_6}
g_6=-\frac{K}{6!},
\end{equation}
with $K$ as in  \eqref{eqn-K}. Incidentally, we recover the value of $\rho$ (which was already determined in Proposition~\ref{prop-evaluation})  from the identity
\[
f_0=f(1/7)=B(1/7)=7/\rho
.\]
 
Moving forward, we make the change of variable $Z=1-7z$, and we write \eqref{B-expansion} as 
\begin{equation}\label{eqn-4}
B(z)=F(Z)+Z^6G(Z)\log(1-\phi(z)), 
\end{equation}
where $F(Z)=f(z)$ and $G(Z)=g(z)/Z^6$ are power series convergent near $Z=0$ with non-vanishing constant terms. We have
\begin{align}
B(z)&=F(Z)+Z^6G(Z)\log(1-\phi(z)) \nonumber \\
&=F(Z)+Z^6G(Z)(\log[(1-\phi(z))/(1-7z)]+\log(1-7z))\nonumber\\
&=F(Z)+Z^6G(Z)(\log[(2z+1)^2/(1-z)^3]+\log(1-7z)) \nonumber\\
&=P(Z)+Z^7\tilde{F}(Z)+Z^6\tilde{G}(Z)+Z^6G(Z)\log Z, \label{eqn-B_prelim}
\end{align}
where
\begin{align*}
 P(Z)&=\sum_{n=0}^6 f_n Z^n,\\
 \tilde{F}(Z)&=(F(Z)-P(Z))/Z^7=\sum_{n\geq7} f_n Z^{n-7},\\
 \tilde{G}(Z)&=G(Z)\log[(2z+1)^2/(1-z)^3].\\
 \end{align*}
Note that
$\tilde{G}(0)=g_6\log({21}/{8})=(-K/6!)\log(21/8)$, and \eqref{eqn-B_prelim} is valid in a neighborhood of~$z=1/7$ in $\tilde{\Omega}$. 

It follows from \eqref{eqn-B_prelim} that
\begin{equation}\label{eqn-B_prelim2}
B(z)=P(Z)+g_6Z^6\log(Z)+Z^7\tilde{F}(Z)+Z^6\tilde{G}(Z)+Z^7H_1(Z)\log Z,
\end{equation}
where $H_1(Z)=(G(Z)-g_6)/Z$. Define the degree-six polynomial $p(Z)$ by 
\begin{align}\label{eqn-polynomial}
p(Z)=P(Z)+Z^6\tilde{G}(0)&=\left(\sum_{n=0}^6f_nZ^n\right)+g_6\log\left(\frac{21}{8}\right)Z^6\\ \nonumber
&=\left(\sum_{n=0}^6f_nZ^n\right)-\frac{K}{6!}\log\left(\frac{21}{8}\right)Z^6.
\end{align}
Setting $H_2(Z)=\tilde{F}(Z)+(\tilde{G}(Z)-\tilde{G}(0))/Z$, 
we obtain \eqref {eqn-B_sing} from \eqref{eqn-g_6} and \eqref{eqn-B_prelim2}. \end{proof}
\subsubsection{Proof of Theorem \ref{thm-conjecture}(c)}
Now we are in a position to verify \eqref{eqn-b_asymptotic} and \eqref{eqn-b_series}. To proceed we need the following lemma.
\begin{lemma}
For $n>k\geq0$,
\begin{equation}\label{eqn-log_taylor}
[x^n][(1-x)^k\log(1-x)]=(-1)^{k+1}\frac{k!}{(n)_k},
\end{equation}
where $(n)_k:=n(n-1)\cdots(n-k).$
\end{lemma}
\begin{proof}
For $k=0$, this is just the expansion $\log(1-x)=-\sum_{n\geq 1} (x^n/n)$. For $k\in\mathbb{N},$ the identity follows by induction.  
\end{proof}
It follows from \eqref{eqn-log_taylor} that 
\[
[z^n]\left(-\frac{K}{6!}Z^6\log Z\right)=\frac{K\cdot7^n}{(n)_6}. 
\]
From this fact, along with \eqref{eqn-B_sing} and Theorem \ref{thm-transfer} in the \nameref{sec-appendix}, we find the following improvement to Proposition \ref{prop-asymptotics}:
\begin{equation}\label{eqn-improvement_b_n}
b_n=K\cdot7^n\left(\frac{1}{n^7}+\mathcal{O}\left(\frac{\log n}{n^{8}}\right)\right),\quad \text{as } n\to\infty. 
\end{equation}

We can take this analysis a step further to expand $b_n$ in an asymptotic series. Indeed, recalling the expansion $g(z)=\sum_{n\geq 6}g_n Z^n$ from above, a careful look at the derivation of \eqref{eqn-B_sing} from \eqref{eqn-B_prelim} in the proof of Proposition \ref{prop-B_sing} shows that $B$ can be written as 
\begin{equation}\label{B-sing_new}
B(z)=\tilde{p}(Z)+g_6Z^6\log Z+g_7Z^7\log Z+Z^8\tilde{H}_2(Z)+Z^8\tilde{H}_1(Z)\log Z, \\
\end{equation}
for 
\begin{align*}
\tilde{p}(Z)&=p(Z)+Z^7H_2(0),\\
\tilde{H}_1(Z)&=(H_1(Z)-g_7)/Z,\\
\tilde{H}_2(Z)&=(H_2(Z)-H_2(0))/Z.
\end{align*}
We have simply extracted the higher order term $g_7Z^7\log Z$ from $g(z)\log(Z)$ to get a more accurate estimate of $B$ near $Z=0$.
In the same way that we derived \eqref{eqn-improvement_b_n}, we find from \eqref{eqn-log_taylor} and \eqref{B-sing_new}  that
\[
b_n=7^n\left(\frac{K}{n^7}+\frac{\tilde{K}}{n^8}+\mathcal{O}\left(\frac{\log n}{n^9}\right)\right),
\]
where $\tilde{K}=7!g_7$. This shows that the $\oh(n^{-8} \log n)$ term in \eqref{eqn-improvement_b_n} is actually $\oh(n^{-8})$, and we obtain the asymptotic expression \eqref{eqn-b_asymptotic}.

One sees moreover that the process can be continued indefinitely, which yields the asymptotic series
\[
\frac{b_n}{7^n}\sim \sum_{i=7}^\infty \frac{K_i}{n^i},
\]
where $K_7=K$, $K_8=\tilde{K}$, and in general $K_n=(-1)^{n}(n-1)!g_{n-1}$.
Since \[\Gamma\left(\frac{2}{3}\right)\Gamma\left(\frac{4}{3}\right)=\frac{1}{3}\Gamma\left(\frac{1}{3}\right)\Gamma\left(\frac{2}{3}\right)=\frac{2\pi}{3\sqrt{3}},\]
as in the proof of Proposition \ref{prop-evaluation}, inspecting the definition of $g(z)$ from \eqref{B-expansion} and the definition of $T_{a,b}$ from \eqref{eqn-1} shows that each $K_i$ is a rational multiple of $\sqrt{3}/\pi$. Thus, with $\kappa_i=K_i\pi/\sqrt{3}$ for $i\geq 7$, we have verified \eqref{eqn-b_series}. 
\subsubsection{Three auxiliary constants}
It will be useful in what follows to define
\begin{equation}\label{eqn-lambda}
\lambda :=(1/7)B'(1/7)=-f_1.
\end{equation}
Of course, one may object that $B$ is not analytic at $1/7$, but \eqref{eqn-B_sing} shows that $B'(1/7)$ still exists as the limit of $B'(z)$, as $z\to 1/7$ in $\tilde{\Omega}$, and has the value $-7p'(0)=-7P'(0)=-7f_1$. Using simplifications like those used to evaluate $\rho$ in Proposition $\ref{prop-evaluation}$,
the value of $\lambda$ is seen to be
\begin{equation*}\label{eqn-lambda_eval}
\lambda=\frac{852768\sqrt{3}-470155\pi}{10\pi}\,\approx \,0.0639.
\end{equation*}
We also record for later that
\begin{equation}\label{eqn-y_'}
y'\left(\frac{1}{7}\right)=\frac{1}{7}B'\left(\frac{1}{7}\right)+B\left(\frac{1}{7}\right)=\lambda+\frac{7}{\rho}\approx 1.0901,
\end{equation}
and that 
\begin{equation}\label{eqn-A_'}
A'\left(\frac{1}{\rho}\right)=7-\frac{49}{\rho\lambda +7}\approx 0.4106.
\end{equation}
The latter can be deduced by differentiating \eqref{eqn-functional_equation_3} and evaluating at $1/7$, 
obtaining
\[y'(1/7)=A'(y(1/7))y'(1/7)+A(y(1/7))=A'(1/\rho)y'(1/7)+7/\rho,\]
and then substituting the expression in \eqref{eqn-y_'}. The $\approx$ symbol indicates an error of less than $10^{-5}$, which could be checked from rational approximations of $\pi$ and $\sqrt{3}$.

\subsection[]{Analytic continuation of $\psi$}\label{subsec-psi}
We introduce the domain 
\[
\Lambda:=\{z:|z|< 1/\rho\}.\]
Observe that $y(\Omega)\subset \Lambda$, by the triangle inequality, since $y_n\geq 0$ for all $n\geq 0$.
Recall from Lemma \ref{lemma-psi} that the function $\psi$, defined by $\psi(z)=z/A(z)$, is analytic on $y(\Omega)$. Also recall from \eqref{eqn-first} that the radius of convergence of $A$ is $1/\rho$, so that, by Pringsheim's Theorem, $1/\rho$ is a singularity of $A$, and so also of $\psi$.

Lemma~\ref{lemma-psi} implies that $(y\circ \psi)|_{y(\Omega)}=\Id|_{y(\Omega)}$. We need to continue $\psi$ analytically to $\Lambda$, and then to a Delta-domain around $\Lambda$,
and we need to maintain this inverse relationship between $y$ and $\psi$. Precisely, we have the following proposition, using terminology and notation from the \nameref{sec-appendix}.

\begin{proposition}\label{prop-psi}
The function $\psi$ is analytically continuable to a Delta-domain $\Delta_{1/\rho}$ around $\Lambda$. Furthermore, 
\[
(y\circ \psi)|_{\Delta_{1/\rho}}=\Id|_{\Delta_{1/\rho}}.
\]
\end{proposition}
In the proof, we will appeal to Lemmas \ref{lem-psi_analytic}, \ref{lem-psi_analytic_2}, and \ref{lemma_z} below.
\begin{lemma}\label{lem-psi_analytic}
$\psi$ is analytic on $\Lambda$ with non-vanishing derivative, and $\psi$ and $\psi'$ extend continuously to $\partial{\Lambda}$.
\end{lemma} 
\begin{proof}
We begin with a numerical evaluation. \eqref{eqn-functional_equation_3} implies that $A(y(1/7))=7/\rho \approx 1.0262,$
where $\rho$ was evaluated in Proposition \ref{prop-evaluation}.

It follows that $\sum_{n\geq 1} a_ny(1/7)^n <1$, and hence that $\sum_{n\geq 1} a_nz^n<1$ for $|z|\leq y(1/7)$, since~$a_n\geq 0$ for all $n$. Since $A(0)=1$, this implies that $A$ does not vanish on $\overline{\Lambda}$. Therefore, $\psi$ is analytic not only on $y(\Omega)$, but also on $\Lambda$, and extends continuously to $\partial\Lambda$, by the Weierstrass M-test applied to $A(z)$.

To verify that $\psi'$ does not vanish on $\Lambda$, we recall equation \eqref{eqn-derivative}:
\[
\psi'(z)=\frac{A(z)-zA'(z)}{A(z)}=\frac{1}{A(z)}\left(1+\sum_{n\geq 1} a_n(1-n)z^n\right).
\]
The quantity $A(z)-zA'(z)$ does not vanish on $(0, y(1/7))$, e.g. by Theorem \ref{thm-criterion} and the fact that~$y_n\sim K7^n/n^7$. Since $a_n \geq 0$, this implies that $
\sum_{n\geq 1} a_n(n-1)z^n<1
$
for all $z\in(0,y(1/7))$, and hence that $\left|\sum_{n\geq 1} a_n(n-1)z^n\right|<1$ for all $z\in\Lambda$. It follows that $\psi'$ does not vanish on $\Lambda$. Moreover,
\eqref{eqn-A_'} shows that $A'(y(1/7))$ is finite, so the Weierstrass M-test implies that $A'(z)$ and $\psi'(z)$ extend continuously to $\partial \Lambda$. 
\end{proof}

Although the lemma extends $\psi$ beyond $y(\Omega)$ to all of $\Lambda$, it does not automatically imply that~$\psi(\Lambda)\subset \tilde{\Omega}$, or, in other words, that $\psi$ avoids the branch cuts $(\infty, -1/2]$ and $[1/7,\infty)$. We resolve this doubt in the following lemma. 
\begin{lemma}\label{lem-psi_analytic_2}
$\psi(\Lambda)\subset\tilde{\Omega}$, $y$ is analytic on $\psi(\Lambda)$,  and $(y\circ\psi)|_{\Lambda}=\Id|_\Lambda$.
\end{lemma}
\begin{proof}
The proof is based on two complementary claims about $\psi$.

\textit{Claim 1}: If $z\in\overline{\Lambda}\setminus{\mathbb{R}}$, then $\psi(z)\not\in \mathbb{R}$.

To see verify the claim, suppose, toward showing the contrapositive, that $\psi(z)\in\mathbb{R}$ for some~$z\in\overline{\Lambda}$. Then $A(z)=\beta z$ for some~$\beta\in \mathbb{R}$.
Observe that $\rho>6.8$ by Proposition \ref{prop-evaluation}, so that $y(1/7)=1/\rho <0.15$, and this implies that
$
6/\rho <2-7/\rho.
$
From $\beta z=1+(A(z)-1)$, we see that
\begin{align*}
|\beta z| &\geq 1-\sup_{w\in \Lambda} |A(w)-1|=1-(A(1/\rho)-1)\\
&=2-7/\rho>6/\rho\,\\
&> 6|z|.
\end{align*}
It follows that $|\beta|>6$. 

On the other hand, since $A(z)=\beta z$, we see that $A(\overline{z})=\overline{A(z)}=\beta \overline {z}$. It follows that
\[\beta(z-\overline{z})=A(z)-A(\overline{z})=\int_{\overline{z}}^z A'(w) \, dw ,\]
where the integration path is the line segment from $\overline{z}$ to $z$, as $\Lambda$ is convex. Therefore,
\[
|\beta|\cdot |z-\overline{z}|\leq |z-\overline{z}|\cdot \sup_{w\in \Lambda} |A'(w)|.
\]
Since $A$ has non-negative Taylor coefficients, $A'$ is bounded by $A'(1/\rho) < 1$, which was evaluated in \eqref{eqn-A_'}. If $z\neq \overline {z}$, this would imply that 
$|\beta| < 1$, which is impossible, since we already showed that $|\beta| >6$.  It follows that $z=\overline{z}$, verifying the claim. 

\textit{Claim 2}:
$\psi$ maps $\overline{\Lambda}\cap\mathbb{R}=[-1/\rho,1/\rho]$ bijectively onto $[\psi(-1/\rho), 1/7]\subset (-1/2, 1/7]$. 

To see why this is true, note that $\psi$ maps the interval $[y(-1/7), 1/\rho]$ bijectively to $[-1/7,1/7]$ by Lemma \ref{lemma-psi}, so it remains only to consider $[-1/\rho,\, y(-1/7))$. 
Since $\psi'(0)=1$, and $\psi'$ is non-vanishing on $\Lambda$ by Lemma~\ref{lem-psi_analytic}, we see that $\psi$ is increasing on $[-1/\rho,\, y(-1/7))$ and maps the latter interval onto 
$[\psi(-1/\rho),-1/7)$, verifying the bijection. To verify the containment~$[\psi(-1/\rho), 1/7]\subset (-1/2, 1/7]$, recall that we approximated $A(1/\rho)$ in the proof of \linebreak Lemma~\ref{lem-psi_analytic}. Here, the estimate $A(1/\rho)<3/2$ is sufficient, since then we have (by the non-negativity of the coefficients of $A$) that $A(-1/\rho)>1/2$, and therefore 
\[
|\psi(-1/\rho)|=\left| \frac{-1/\rho}{A(-1/\rho)} \right|= \frac{1/\rho}{A(-1/\rho)}<2/\rho <1/2,
\]
which verifies the claim. 

Taken together, these two claims show that $\psi(\Lambda)\subset\tilde{\Omega}$. It follows that $y$ is analytic on $\psi(\Lambda)$. Furthermore, $(y\circ\psi)|_{\Lambda}=\Id|_\Lambda$, by the principle of permanence, since $(y\circ\psi)|_{y(\Omega)}=\Id|_{y(\Omega)}$ and $y(\Omega)\subset \Lambda$. 
\end{proof}

Having established the analytic continuation of $\psi$ to $\Lambda$ in the last two lemmas, we proceed now with the the continuation of $\psi$ to a Delta-domain $\Delta_{1/\rho}$ around $\Lambda$. The first step is to continue~$\psi$ near $y(1/7)=1/\rho$. For $z\in\mathbb{C}$, $\epsilon>0$, and $\delta \in (0, \pi/2)$, define the domain
\[
D_{\epsilon,\delta}(z):=\{w:|w-z|<\epsilon, |\Arg(w-z)| >\delta\},
\]
where $\Arg$ denotes the principal value of the argument. We claim that $\psi$ can be continued analytically to $D_{\epsilon,\delta}(1/\rho)$, for some $\epsilon>0$ and some $\delta\in (0,\pi/2)$. The basic idea is that the map~$z\mapsto y(z)$ acts approximately as a dilation near $z=1/7$, and its inverse $y^{-1}$ locally extends~$\psi$ to a domain of the desired form. We make this precise in the following lemma.

\begin{lemma}\label{lemma_z}
There exist $\epsilon>0$ and $\epsilon'>0$ such that $D_{\epsilon',3\pi/8}(1/\rho)\subset y(D_{\epsilon, \pi/4}(1/7))$. Moreover,~$\psi$ is analytic on $D_{\epsilon',3\pi/8}(1/\rho)$, and
$(y\circ \psi)|_{D_{\epsilon',3\pi/8}(1/\rho)}=\Id|_{D_{\epsilon',3\pi/8}(1/\rho)}$.
\end{lemma}
\begin{proof}

Looking back at \eqref{eqn-B_sing}, we see that, despite the presence of higher-order logarithmic terms, there is a simple linear approximation for $B(z)$ as $z\to 1/7$ in $\tilde{\Omega}$. In particular,
\[
B(z)=f_0+7f_1(1/7-z) +\oh((1/7-z)^2),
\]
where $f_0$ and $f_1$ were defined in the proof of Proposition \ref{prop-B_sing}. Recall that $1/\rho=y(1/7)=B(1/7)/7=f_0/7$, and $y'(1/7)=f_0-f_1$.
Therefore, the following approximation holds for $y$:
\begin{align}\label{new-equation}
y(1/7)-y(z)&=y(1/7)-zB(z) \nonumber\\
&=f_0/7+(1/7-z)B(z)-B(z)/7\nonumber\\
&=f_0/7+f_0(1/7-z)-(f_0/7+f_1(1/7-z))+(\oh(1/7-z)^2)\nonumber\\
&=(1/7-z)(y'(1/7)+\oh(1/7-z)) 
\end{align}

As a consequence, there exists $\epsilon>0$
such that
the following properties hold for all $z\in\tilde{\Omega}$ such that $|z-1/7|<\epsilon$, and in particular on $D_{\epsilon,\pi/4}(1/7)$:
\begin{itemize}
\item[(1)] $\frac{8y'(1/7)}{9}<\frac{|y(z)-y(1/7)|}{|z-1/7|}<\frac{10y'(1/7)}{9}$
\item[(2)] $\left|\Arg\left(\frac{y(1/7)-y(z)}{1/7-z}\right)\right|<\pi/8$
\item[(3)] $|y'(z)-y'(1/7)| < \frac{y'(1/7)}{9}$
\end{itemize}

The first two properties follow from \eqref{new-equation}, as the $\oh(1/7-z)$ term can be bounded in modulus by $y'(1/7)/9$ (implying (1)), which is less than $y'(1/7)\tan(\pi/8)$ (implying (2)). The third property follows from the continuity of $y'$, as discussed before \eqref{eqn-y_'}. Properties (1) and~(2) will help us to understand the shape of $y(D_{\epsilon,\pi/4}(1/7))$, while (3) will help us to argue that $y$ is injective on $D_{\epsilon,\pi/4}(1/7)$.

Let $\gamma$ be the straight line emanating from $1/7$, of the form $\gamma(t)=1/7+te^{i\pi/4}$, for $t\geq 0$.
Property (2) implies that
\[
\pi/8<\Arg(y(\gamma(t))-y(1/7))<3\pi/8
\]
 for $t\in(0,\epsilon)$. Similarly, we also have $-\pi/8>\Arg(y(\overline{\gamma}(t))-y(1/7))>-3\pi/8$ for $t\in(0,\epsilon)$, where $\overline{\gamma}(t)=1/7+te^{-i\pi/4}$. Moreover, property (1) implies, by continuity, that for $z\in\tilde{\Omega}$ such that $|z-1/7|=\epsilon$, the following is true:
\[
\frac{8y'(1/7)\epsilon}{9}\leq |y(z)-y(1/7)|\leq \frac{10y'(1/7)\epsilon}{9}\,
\]

Thus, we have described the images of the two line segments and the circular arc that make up the boundary of $D_{\epsilon, \pi/4}(1/7)$. Our description implies that the image $y(D_{\epsilon, \pi/4}(1/7))$ contains~$D_{\epsilon', \delta}(1/\rho)$, for $\epsilon' = 8y'(1/7)\epsilon/9$, and $\delta=3\pi/8$, as shown in Figure \ref{fig-D}. This proves the first statement in the lemma.

\begin{figure}
\centering
\includegraphics[scale=0.65]{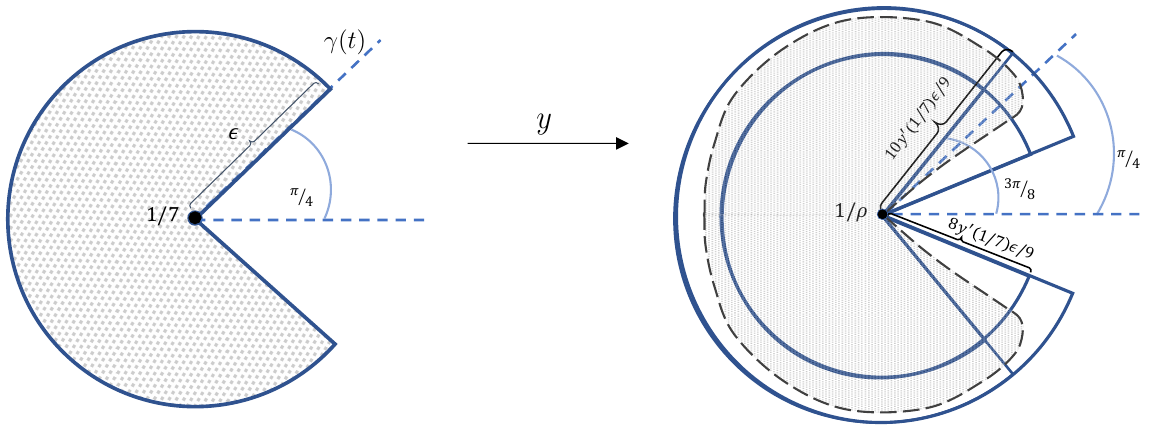}
\caption{The line $\gamma$ and the region $D_{\epsilon,\pi/4}(1/7)$ (left), with its image (shaded, right).}
\label{fig-D}
\centering
\end{figure}

To prove the second part of the lemma, we argue that $y$ is injective on $D_{\epsilon, \pi/4}(1/7)$.
Let $z_1$ and $z_2$ be two points in $D_{\epsilon, \pi/4}(1/7)$. If the line segment connecting $z_1$ and $z_2$ lies in $D_{\epsilon, \pi/4}(1/7)$, then we estimate $|y(z_1)-y(z_2)|$ as follows. Since
\[y(z_1)-y(z_2)=\int_{z_2}^{z_1} y'(w)\, dw= \int_{z_2}^{z_1} y'(1/7)\,dw +\int_{z_2}^{z_1}[y'(w)-y'(1/7)]\, dw,\]
where the contour of integration is the line segment, it follows from property (3) above that
\begin{align*}\nonumber\label{eqn-estimate_1}
|y(z_1)-y(z_2)|&\geq \left|\int_{z_2}^{z_1}y'(1/7)\, dw\right| - \left |\int_{z_2}^{z_1}[y'(w)-y'(1/7)]\, dw\right|\\ \nonumber
&\geq y'(1/7)|z_1-z_2| - \frac{y'(1/7)}{9}|z_1-z_2|\\\nonumber
&= \frac{8y'(1/7)}{9} |z_1-z_2|.
\end{align*}
Thus, we see that $|y(z_1)-y(z_2)|=0$ implies that $z_1=z_2$.

In the case that the line segment connecting $z_1$ and $z_2$ does not lie in $D_{\epsilon, \pi/4}(1/7)$, then $z_1$ and $z_2$ must have non-zero imaginary parts of opposite sign, so that the line segment connecting~$z_1$ and $\overline{z_2}$ does indeed lie in $D_{\epsilon, \pi/4}(1/7)$. Moreover, it must be that ${|\Arg(z-1/7)|\in(\pi/4,3\pi/4)}$ for $z \! = \! z_1$ and $z \! = \! \overline{z_2}$, and so for all other points $z$ on this line segment. Suppose that~${y(z_1) \! = \! y(z_2)}$. Then $y(z_1) =\overline{y(\overline{z_2})}$, by the Schwarz Reflection Principle, and so the imaginary part of $y(z)$ takes opposite values at the two endpoints of the line segment connecting $z_1$ and $\overline{z_2}$. Since the imaginary part of $y(z)$ varies continuously as $z$ moves along this line segment, there is a point $w$ on the line segment at which $y(w)$ is real. 
It follows that
\[
\left|\Arg\left(\frac{y(1/7)-y(w)}{1/7-w}\right)\right|=|\Arg(1/7-w)|\in (\pi/4,3\pi/4),
\]
which contradicts property (2). 

We conclude that $y$ is injective on $\tilde{\Omega}\cap D_{\epsilon,\pi/4}(1/7)$ and admits an inverse there, which we call~$y^{-1}$. Observe that $y^{-1}$ is analytic on $y(D_{\epsilon,\pi/4}(1/7))$, as the inverse of an analytic map \cite[Ch. 9.4]{Hille}, and hence $y^{-1}$ is analytic on the subset $D_{\epsilon',3\pi/8}(1/\rho)$. Moreover, $y^{-1}$ agrees with $\psi$ on~$\Lambda\cap D_{\epsilon',3\pi/8}(1/\rho)$,  by Lemma \ref{lem-psi_analytic_2}. Consequently, we have, by the principle of permanence, the inverse relationship $(y\circ \psi)|_{D_{\epsilon',3\pi/8}(1/\rho)}=\Id|_{D_{\epsilon',3\pi/8}(1/\rho)}$, as was to be shown.
\end{proof}

\begin{proof}[Proof of Proposition \ref{prop-psi}]

To finish proving that $\psi$ can be analytically continued to $\Delta_{1/\rho}$, it suffices, by the compactness of the set
$\partial \Lambda \setminus {D_{\epsilon',3\pi/8}(1/\rho)}$ and in light of Lemma \ref{lemma_z},
to demonstrate the existence of an analytic continuation of $\psi$ around each point in an arbitrary finite subcollection of $\partial \Lambda\setminus{\{1/\rho\}}$.

Let $\tilde{w}\in \partial \Lambda \setminus{\{1/\rho\}}$. The two supporting claims in the proof of Lemma \ref{lem-psi_analytic_2} imply \linebreak that~${\psi(\tilde{w})\in\tilde{\Omega}}$. 
We claim that $y'(\psi(\tilde{w}))\neq 0$. To see why this is true, let $\gamma$ be the straight line path from $0$ to $\tilde{w}$.
If $y'(\psi(\tilde{w}))=0$, then, recalling from Lemma \ref{lem-psi_analytic_2} that $y\circ\psi=\Id|_\Lambda$, we see that
$y(\psi(w))=w$ for $w\in \gamma$. Therefore,
\[0=\lim_{w\to \tilde{w}, w\in\gamma} |y'(\psi(w))|=\lim_{w\to \tilde{w}, w\in\gamma} 1/|\psi'(w)|,\]
which implies that $|\psi'(w)|\to \infty$ as $w\to\tilde{w}$ along $\gamma$. This contradicts the fact from Lemma \ref{lem-psi_analytic} that $\psi'$ extends continuously to $\partial \Lambda$ and so is bounded there. So $y'(\psi(\tilde{w}))\neq0$, as claimed.

By the inverse function theorem, $y$ is locally invertible at $\psi(\tilde{w})$, and the local inverse extends~$\psi$ analytically to a neighborhood of $\tilde{w}$ that maps into $\tilde{\Omega}$. 

We have extended $\psi$ to a Delta-domain $\Delta_{1/\rho}$ by patching together \textit{local} inverses of $y$. \linebreak To see that $y$ is a \textit{global} inverse of $\psi$, one need only observe that $\psi(\Delta_{1/\rho})\subset\tilde{\Omega}$, by construction. \linebreak It follows that $y$ is analytic on $\psi(\Delta_{1/\rho})$, and,  by the principle of permanence, \linebreak that ${(y\circ \psi)|_{\Delta_{1/\rho}}=\Id|_{\Delta_{1/\rho}}}$. 
\end{proof}

\subsection[]{Asymptotic singular expansion of $\psi$ near $y(1/7)$}
Let $\Delta_{1/\rho}$ be the Delta-domain from Proposition \ref{prop-psi}. 
We apply a bootstrapping procedure to \eqref{eqn-B_sing} to derive an asymptotic singular expansion for $\psi=y^{-1}$ near the singularity $1/\rho$. The general idea of bootstrapping as well as some examples are discussed in \cite[Ch. 2]{DeBruijn}. Precisely, we show the following, with $\lambda$ as in \eqref{eqn-lambda}.
\begin{proposition} Set $V=1-\rho z$. The function $\psi$ admits the following singular expansion: as~$z\to 1/\rho$ in $\Delta_{1/\rho}$, we have\label{prop-psi_sing}
\begin{equation} \label{eqn-psi-sing}
\psi(z)=\gamma(V)+CV^6\log V+\oh(V^7 \log V), 
\end{equation}
where
\[
C=\frac{7^5K\rho}{6!(7+\rho \lambda)^7} ,
\]
and $\gamma$ is a degree-six polynomial with $\gamma(0)=1/7$. 
Furthermore, the error term $\psi(z)-\gamma(V)-CV^6\log V$ is analytic in $\Delta_{1/\rho}$.
\end{proposition}
\begin{proof}
With $Z=1-7z$, \eqref{eqn-B_sing} implies that
\[
y(z)=\frac{1-Z}{7}\left(p(Z)-\frac{K}{6!}Z^6\log Z+\mathcal{O}(Z^7\log Z)\right)
,
\]
as $z\to 1/7 \in \tilde{\Omega}$. By the identity $(y\circ \psi)|_{\Delta_{1/\rho}}=\Id|_{\Delta_{1/\rho}}$, if we evaluate both sides of the equation at $\psi(z)$, for $z$ in a small neighborhood of $1/\rho$ in $\Delta_{1/\rho}$,
and if we apply the change of variables~$Y=1-7\psi(z)$, then we obtain
\[
\frac{1-Z}{7}=z=y(\psi(z))=\frac{1-Y}{7}\left(p(Y)-\frac{K}{6!}Y^6\log(Y)+\mathcal{O}(Y^7\log Y)\right),
\]
which implies
\begin{equation}\label{eqn-Z}
Z-1=(Y-1)\left(p(Y)-\frac{K}{6!}Y^6\log Y+\oh(Y^7\log Y)\right),
\end{equation}
as $z\to 1/\rho$ in $\Delta_{1/\rho}$. 
In \eqref{eqn-polynomial} we see that $p(Y)=(-K/6!)\log(21/8)Y^6+\sum_{n=0}^6 f_n Y^n$. Thus,
upon setting
\[W=\frac{Z+f_0-1}{f_0-f_1},\] we find from \eqref{eqn-Z} that as $z\to 1/\rho$ in $\Delta_{1/\rho}$, 
\begin{equation}
\label{eqn-sub}Y=W-(Q(Y)+cY^6\log Y+\oh(Y^7\log Y)),
\end{equation}
where
{
\small
\begin{align*}
Q(Y)= \frac{{f_{1} - f_{2}}}{f_{0} - f_{1}}Y^2 +\frac{{f_{2} - f_{3}}}{f_{0} - f_{1}}Y^3 + \frac{{f_{3} - f_{4}}}{f_{0} - f_{1}}Y^4 +  \frac{{f_{4} - f_{5}}}{f_{0} - f_{1}}Y^5 +\frac{f_{5} - f_{6}+\frac{K}{6!}\log\left(\frac{21}{8}\right) }{f_{0} - f_{1}}Y^6,
\end{align*}
}
and
\begin{equation*}\label{eqn-c}
c=\frac{K}{6!(f_0-f_1)}=\frac{K}{6!(7/\rho+\lambda)}.
\end{equation*}
To simplify notation, define the list of constants $(a_i)_{i=2}^6\subset \mathbb{R}$ so that $Q(Y)=\sum_{i=2}^6 a_i Y^i$.

$W(z)$ and $Y(z)$
tend to $0$ as $z\to 1/\rho$ in $\Delta_{1/\rho}$. It follows then from \eqref{eqn-sub} that $W\sim Y$ as~$z\to 1/\rho$.
 Therefore, $Y=\oh (W)$ and $\log(Y)\sim \log(W)$ (as $z\to 1/\rho$ in $\Delta_{1/\rho}$, which is assumed for the rest of the proof).
These estimates for $Y$ imply  that
\[Q(Y)=a_2W^2+\oh(W^3),\]
which, upon substitution into \eqref{eqn-sub}, yields
\[
Y=W-a_2W^2+\oh(W^3).
\]

We substitute this new estimate for $Y$ into $Q(Y)$, which yields
\[Q(Y)=a_2W^2+(a_3-2a_2^2)W^3+\oh(W^4),\]
and hence, 
from \eqref{eqn-sub},
\[
Y=W-a_2W^2-(a_3-2a_2^2)W^3+\oh(W^4).
\]
After another iteration, we obtain from \eqref{eqn-sub} that
\begin{align*}
Y&=W-Q(Y)+\oh(W^5)\\
&= W- a_2W^2-(a_3-2a_2^2)W^3-(5a_2^3-5a_2a_3+a_4W^4)+\oh(W^5).
\end{align*}
After a final iteration, we obtain from \eqref{eqn-sub} that
\begin{equation*}\label{eqn-bootstrapping}
Y=W-a_2W^2-\cdots-a_6'W^6+ \oh(W^7)-cY^6\log Y +\oh(Y^7\log Y),
\end{equation*}
where $a_6'$ is an unimportant constant.
Since $Y=W+\oh(W^2)$, we have that the 
$\oh(Y^7\log Y)$ error term
is $\oh (W^7\log W)$, and that $Y^6\log Y=W^6\log W+\oh(W^7\log W)$.
We can therefore express $Y$ in terms of $W$ as follows:
\[
Y=W-a_2W^2-\cdots-a_6'W^6-cW^6\log W +\oh(W^7\log W).
\]
We substitute $1-7\psi(z)$ for $Y$ in the left side of the latter expression, obtaining
\begin{equation}\label{eqn-new_psi}
\psi(z)=
\frac{1}{7}+\tilde{P}(W)+\frac{c}{7}W^6\log(W)+\oh(W^7\log W),
\end{equation}
where $\tilde{P}(W)$ is a degree-six polynomial with $\tilde{P}(0)=0$. 

To finish the proof, it only remains to rewrite \eqref{eqn-new_psi} with $W$ in terms of $V$:
\[
W=\frac{Z+f_0-1}{f_0-f_1}=\frac{Z+7/\rho-1}{7/\rho+\lambda}=\frac{7/\rho-7z}{7/\rho+\lambda}=\frac{7}{7+\rho \lambda}(1-\rho z)=\frac{7}{7+\rho \lambda}V.
\]
Substituting the last expression into \eqref{eqn-new_psi} leads immediately to \eqref{eqn-psi-sing}, with
\begin{equation}\label{gamma}\nonumber
\gamma(V)=\frac{1}{7}+\tilde{P}\left(\frac{7}{7+\rho \lambda}V\right),
\end{equation}
and
\begin{equation}\label{C}\nonumber
C=\frac{c}{7}\left(\frac{7}{7+\rho \lambda}\right)^{6}=\frac{7^5K\rho}{6!(7+\rho \lambda)^7}.
\end{equation}
The error term in \eqref{eqn-new_psi} is analytic in $\Delta_{1/\rho}$ since all the other terms are. 
\end{proof}

\subsection[]{Asymptotic singular expansion of $A$ near $y(1/7)$}
Since $\psi$ is injective on $\Delta_{1/\rho}$ by Proposition \ref{prop-psi}, and $\psi(0)=0$, we see that $A(z)=z/\psi(z)$ extends to be analytic on $\Delta_{1/\rho}$. We seek its singular expansion near $1/\rho=y(1/7)$. Let $C$ be the constant from Proposition \ref{prop-psi_sing}.
\begin{proposition}\label{prop-A_sing} Set $V=1-\rho z$. The function $A$ admits the following singular expansion:
as~$z\to 1/\rho$ in $\Delta_{1/\rho}$,
\begin{equation}\label{eqn-A_sing1}
A(z)=\eta(V)-\frac{49C}{\rho}V^6\log V+\oh(V^7\log V),
\end{equation}
where $\eta$ is a degree-seven polynomial.
The error term 
$A(z)+(49C/\rho)V^6\log V-\eta(V)$ is analytic in $\Delta_{1/\rho}$.
\end{proposition}
\begin{proof}
Since $A(z)=z/\psi(z)$, we need to obtain a useful expression for the reciprocal of $\psi$. Referring to Proposition \ref{prop-psi_sing}, let $E(V)$ denote the error term in \eqref{eqn-psi-sing}, namely,
\[
E(V):=\psi(z)-\gamma(V)-CV^6\log V.
\] 
$E(V)$ is visibly analytic for $z\in\Delta_{1/\rho}$.

Next, define $T(V)$ to be the degree-six Taylor approximation polynomial for $1/\gamma(V)$ centered at $V=0$, i.e. $\gamma(V)T(V) = 1 + \oh(V^7)$ as $V\to 0$. 
Also define $S(V):=T(V)-49CV^6\log V$, which is analytic on $\Delta_{1/\rho}$. Noting that $\gamma(0) = 1/7$ and $T(0) = 7$, we have
\begin{align}\label{psi-S}
\psi(z)S(V)&=(\gamma(V)+CV^6\log V+E(V))(T(V)-49CV^6\log V)\nonumber\\
&=1+E(V)T(V)-49CV^6\log V\cdot E(V) +\oh(V^7 \log V)\nonumber\\
&=1+\oh(V^7 \log V)
\end{align}
as $V\to 0$.
Note that the $\oh(V^7\log V)$ error term $\tilde{E}(V):=\psi(z)S(V)-1$ is analytic on $\Delta_{1/\rho}$. 

We rewrite \eqref{psi-S} as $1/\psi(z) = S(V)-\tilde{E}(V)/\psi(z)$.
Since $z=(1-V)/\rho$, this can in turn be rewritten as follows:
\begin{align*}
A(z)=\frac{z}{\psi(z)}&=\frac{1-V}{\rho}S(V)-A(z)\tilde{E}(V)\\
&=\frac{1-V}{\rho}(S(V)+49CV^6\log V)-\frac{1-V}{\rho}49CV^6 \log V-A(z)\tilde{E(V)}\\
&=\frac{1-V}{\rho}T(V) - \frac{49CV^6\log V}{\rho} -A(z)\tilde{E}(V)+ \frac{49CV^7 \log V}{\rho}\\
&=\eta(V)-\frac{49C}{\rho}V^6\log V +\tilde{\tilde{E}}(V),
\end{align*}
where $\eta(V) := (1-V)T(V)/\rho$ is a degree-seven polynomial, and
\[
\tilde{\tilde{E}}(V):=\frac{49C}{\rho}V^7\log (V) - A(z)\tilde{E}(V)
\]
is analytic in $\Delta_{1/\rho}$. Moreover, $\tilde{\tilde{E}}(V)= \oh(V^7\log V)$ as $z\to 1/\rho$, since the same is true of $\tilde{E}(V)$ and since $A$ is continuous at $1/\rho$. This establishes \eqref{eqn-A_sing1}.
\end{proof}

\subsection[]{Asymptotic formula for $(a_n)$}
Looking at \eqref{eqn-A_sing1}, we see that the polynomial $\eta(V)$ does not contribute to the asymptotics of the sequence $(a_n)$ of Taylor coefficients of $A(z)$. The asymptotic expression for $a_n$ is determined by the other two summands, in accordance with Theorem \ref{thm-transfer}. For each positive integer $n$, the coefficient of $z^n$ in the Taylor expansion of $(-46C/\rho)V^6\log V$ is shown by Lemma \ref{eqn-log_taylor} to be
\begin{align*}
-\frac{49C}{\rho}[z^n][(1-\rho z)\log(1-\rho z)] &= -\frac{49C}{\rho} \cdot \frac{-6!\rho^n}{n(n-1)\cdots (n-6)}\\
&=\frac{49C\cdot 6!}{\rho}\cdot \rho^n\cdot \left(\frac{1}{n^7}+\oh\left(\frac{1}{n^8}\right)\right),
\end{align*}
as $n\to\infty$.
Note that $(49C\cdot6!)/\rho$ is precisely the constant $M$ from \eqref{eqn-M}, as we computed with SAGE. By Theorem \ref{thm-transfer}, the $\oh(V^7\log V)$ term in \eqref{eqn-A_sing1} contributes the term $\oh(\rho^nn^{-8}\log n)$ to the asymptotic expression for $a_n$. We thus obtain the following:
\begin{align*}
a_n&=M\rho^n\left(\frac{1}{n^7}+\mathcal{O}\left(\frac{1}{n^{8}}\right)\right)+\mathcal{O}\left(\frac{\rho^n\log n}{n^{8}}\right)\\
&=M\rho^n\left(\frac{1}{n^7}+\mathcal{O}\left(\frac{\log n}{n^{8}}\right)\right),
\end{align*}
as $n\to\infty$. This verifies \eqref{eqn-a_asym} and completes the proof of Theorem~\ref{thm-conjecture}.


\section*{Appendix}
\label{sec-appendix}
We record here, as Theorem \ref{thm-transfer}, several ``transfer theorems" of Flajolet and Odlyzko \cite{Odlyzko}, which we use to transfer asymptotic growth estimates of a function $f$ near a singularity to asymptotic growth estimates for the function's Taylor coefficients $(f_n)_{n\geq 0}$.

Assume the following setup: $f$ is a function analytic at the origin, with radius of convergence~$R>0$ and Taylor expansion $f(z)=\sum_{n=0}^\infty f_nz^n$, and $f$ can be continued analytically to a ``Delta-domain" $\Delta_R$, which is defined to be any open set of the form
\[\{z:|z|<R+\epsilon, |\operatorname{Arg}(z-R)|>\theta\},\] for some $\epsilon>0$ and some $\theta\in(0,\pi/2)$. Here, $\Arg$ denotes the principle value of the argument. In particular, $f$ has a unique singularity on its disk of convergence, namely the point $R$.
We may refer to $\Delta_R$ as a ``Delta-domain around the disk of convergence of $f$."
Often in practical applications of these transfer theorems, a function will admit analytic continuation to well-beyond a Delta-domain, e.g. to $\mathbb{C}\setminus{[R,\infty)}$.

Let $\log$ denote the principle branch of the logarithm.

\begin{theorem}{\cite[Cor. 2 and Thm. 2]{Odlyzko} }\label{thm-transfer}
Assume that $f$ is analytic in a Delta-domain $\Delta_R$. 
\begin{enumerate}

\item
If $f(z)\sim K(1-z/R)^\alpha$ as $z\to 1$ in $\Delta_R$, where $\alpha\in \mathbb{C}\setminus{\{0,1,2,\dots\}}$, and $K\in\mathbb{C}$, then
\[
f_n\sim \frac{1}{R^n}\cdot \frac{K}{\Gamma(-\alpha)}n^{-\alpha-1}, \quad\text{ as }\, n\to\infty. 
\]
\item
If $f(z)=\oh((1-z/R)^\alpha(\log(1-z/R))^\gamma)$ as $z\to R$ in $\Delta_R$, where $\alpha,\gamma\in\mathbb{R}$, then 
\[
f_n=\oh\left(\frac{n^{-\alpha-1}(\log n)^\gamma}{R^n}\right), \quad\text{ as }\, n\to\infty. 
\]
\item As a consequence, if $F(z)=g(z)+f(z)$, for $f$ and $g$ analytic on $\{z: |z|<R\}$, and furthermore $f(z)$ is analytic in a Delta-domain $\Delta_R$ and satisfies 
\[f(z)=\oh\big((1-z/R)^\alpha(\log(1-z/R))^\gamma\big),\] 
as $z\to R$ in $\Delta_R$,
then
\[
F_n=g_n+f_n=g_n+\oh\left(\frac{n^{-\alpha-1}(\log n)^\gamma}{R^n}\right), \quad\text{ as }\, n\to\infty.\] 
\end{enumerate}
\end{theorem}
Theorem \ref{thm-transfer} is used in Section \ref{sec-improved} to obtain asymptotic estimates for $(a_n)$ and $(b_n)$, with~${\gamma \! = \! 1}$ in both cases. 
The theorem is actually stated and proved in \cite{Odlyzko} for the special case $R=1$, and the adjustments for $R>0$ are straightforward.


\section*{Acknowledgements}

The author would like to thank Dan Romik, Greg Kuperberg, Michael Drmota, and Hamilton Santhakumar for their encouragement and for many helpful discussions. The author would also like to thank several anonymous referees for their careful suggestions.


\bibliographystyle{alphaurl}
\bibliography{ct-sample-3.bib}

\end{document}